\documentclass[12pt,reqno]{amsart}
\usepackage{amsfonts,latexsym,amssymb,amsmath,amsthm}
\usepackage{graphicx,hyperref,cite}
\newtheorem{lemma}{\bf{Lemma} }[section]
\newtheorem{proposition}{\bf{Proposition}}[section]
\newtheorem{theorem}{\bf{Theorem}}[section]
\newtheorem{remark}{\sc{Remark} }[section]
\newtheorem{definition}{\sc{Definition} }[section]
\newtheorem{corollary}{\bf{Corollary} }[section]

\begin{document}

\title{Radiative effects on the thermoelectric problems}
\author{Luisa Consiglieri}
\address{Luisa Consiglieri, Independent Researcher Professor,  European Union}
\urladdr{\href{http://sites.google.com/site/luisaconsiglieri}{http://sites.google.com/site/luisaconsiglieri}}

\begin{abstract}
There are two main directions in this paper.
One is  to find  sufficient conditions to ensure the existence
of weak solutions to thermoelectric problems.
 At the steady-state, these problems consist
 of a coupled system of elliptic equations of the divergence form,
 commonly accomplished with nonlinear radiation-type conditions
on at least a nonempty part of the  boundary of a $C^1$ domain.
The model under study takes
the   thermoelectric Peltier and Seebeck effects into account,
which describe the Joule-Thomson effect. 
 The proof method makes recourse of a fixed point argument.
To this end, well-determined estimates are our main concern.
  The paper  is in the second direction for
the derivation of explicit $W^{1,p}$-estimates 
$(p>2)$ for  solutions of nonlinear radiation-type problems, 
where the leading coefficient
 is assumed to be a discontinuous function on the space
variable.
In particular, the behavior of the leading coefficient is 
conveniently explicit on the estimate of any solution.
\end{abstract}


\maketitle

\section{Introduction}

This paper is concerned with a model on thermoelectric devices with
radiative effects. 
We formulate the problem by coupling a thermal model with a electrical
model. The work is two-fold. Firstly,
the first part (and Appendix) is of physical nature. Secondly,
Sections \ref{state}, \ref{reverse}, \ref{schigh}, and
\ref{sce} are of mathematical nature and are devoted to obtaining the existence
result for the proposed model (Section \ref{secm2}).
We prove the higher integrability of the gradient of weak solutions to the 
boundary value problem under study.
 This regularity result is sufficiently general to contribute to other problems, in which the dependence on the values of the involved constants is essential, instead of the problem under study only.
 Although the techniques used in the paper are standard (see 
\cite{ak,ark99,bf,cos,drey,gs,gia83,gong,nw}
and the references therein), 
the explicit expressions of the involved constants as function of the data are new.
The derivation of $W^{1,p}$-estimates in \cite{cp,shen} makes use of the
contradiction argument which invalidates the determination of 
explicit expressions of the involved constants as function on the data.
Our final result (Theorem \ref{main2}) is to derive sufficient conditions
on the data to ensure the existence of at least one weak solution
to the thermoelectric problem under study in two-dimensional
real space. Similar work on the 3D existence remains an open problem.
We mention to \cite{kaiser} the optimal elliptic regularity
for the spatio-material model constellations in three-dimensional
real space.

Let $\Omega$ be a bounded domain  (that is, connected open set)
in $\mathbb{R}^n$ ($n\geq 2$),
 according to Definition \ref{domega},
representing a thermoelectric conductor material, which is a
heterogeneous anisotropic solid.
 Assume  that $\Omega$ is of class $ C^{1}$ (cf. Definition \ref{cka}).

Let us consider the following problem that extends the thermoelectric problems,
which were introduced in \cite{zamm,aduf}, in the sense of that
 the thermoelectric coefficient
 is assumed to be a given but arbitrary nonlinear function.
The electrical current density $\bf j$ and  the energy flux density
${\bf J}={\bf q}+\phi{\bf j}$, with 
$\bf q$ being the heat flux vector, 
 satisfy
\begin{equation}\label{jj}
\left\{\begin{array}{ll}\nabla\cdot{\bf j}=0 &\mbox{ in }\Omega\\
-{\bf j}\cdot{\bf n}=g &\mbox{ on }\Gamma_{\rm N}\\
{\bf j}\cdot{\bf n}=0 &\mbox{ on }\Gamma
\end{array}\right.\qquad
\left\{\begin{array}{ll}
\nabla\cdot{\bf J}=0 &\mbox{ in }\Omega\\
{\bf J}\cdot{\bf n}=0 &\mbox{ on }\Gamma_{\rm N}\\
-{\bf J}\cdot{\bf n}=f_\lambda(\theta)
|\theta|^{\ell-2}\theta-\gamma(\theta) \theta_\mathrm{e}^{\ell-1}
 &\mbox{ on }\Gamma,
\end{array}\right.
\end{equation}
for $\ell\geq 2$.
Here $\bf n$ is the unit outward normal to the boundary $\partial\Omega$,
 $g$ denotes the surface current source,
 $f_\lambda$ is a temperature dependent function that
expresses the radiation law 
depending on the wavelength $\lambda$,
and   $\gamma  \theta_\mathrm{e}^{\ell-1}$ stands for the  external heat sources,
with $ \theta_\mathrm{e}$ being an external temperature.
The Kirchhoff radiation law, its variants, and extensions are analyzed
in \cite{lap,land,lorentz,rybi} for real physical bodies.
According to the Stefan--Boltzmann radiation  law, $\ell=5$,
$f_\lambda(T)=\sigma_{\rm SB}\epsilon (T)$ and 
$\gamma(T)=\sigma_{\rm SB} \alpha(T)$, where
$\sigma_{\rm SB}= 5.67\times 10^{-8}$W$\cdot $m$^{-2} \cdot$K$^{-4}$
 is the Stefan-Boltzmann constant for blackbodies.
The parameters, the emissivity $\epsilon$ and the absorptivity $\alpha$, both
 depend on the spatial variable
and the temperature function $\theta$.
If $\ell=2$, the boundary condition corresponds to the Newton law
of cooling with heat transfer coefficient $f_\lambda=\gamma$.

The constitutive equations of state,
\begin{align}\label{pheno1}
{\bf q}&= -k\nabla\theta-\Pi\sigma\nabla\phi;\\
{\bf j}&= -\alpha_{\rm s}\sigma\nabla\theta-\sigma\nabla\phi,\label{pheno2}
\end{align} 
are based on the principle of local thermodynamic
equilibrium of physically small subsystems
(see \cite{goupil} and the references therein).
Here $\theta$ denotes the absolute temperature, 
$\phi$ is the electric potential, 
$\alpha_{\rm s}$ represents
the Seebeck coefficient,
and the Peltier coefficient
$\Pi(\theta)=\theta\alpha_{\rm s}(\theta)$ is
due to  the first Kelvin relation \cite{sarro}.
The electrical conductivity $\sigma$, and
the thermal conductivity $k=k_{\rm T}+\Pi\alpha_{\rm s}\sigma$, with $k_T$ denotes the
purely conductive contribution,
are, respectively, the known positive coefficients of Ohm and Fourier laws.
Both coefficients  depend on the spatial variable
and the temperature function $\theta$ \cite{rou11},
 which invalidates, for instance, the use of the Kirchhoff transformation.

For $p>1$, if meas$(\Gamma)=0$ let the reflexive Banach space
\[V_{p}:=\{ v\in
W^{1,p}(\Omega):\ \int_\Omega v\mathrm{dx}=0\}
\]
 endowed with the seminorm of $W^{1,p}(\Omega)$.
 
For $p>1$, and $\ell\geq 1$, if meas$(\Gamma)>0$ 
let the reflexive Banach space \cite{dpz}
\[V_{p,\ell}:=\{ v\in
W^{1,p}(\Omega):\ v\in L^{\ell}(\Gamma)\}
\]
 endowed with the norm 
 \[
\| v\|_{1,p,\ell}:=\|\nabla  v\|_{p,\Omega}+\|v\|_{\ell,\Gamma}.
\]
For the sake of simplicity, we denote by the same designation $v$
the trace of a function  $v\in W^{1,1}(\Omega)$.
Observe that $V_{p,\ell}$ is a Hilbert space  equipped with the inner product
only if $p=\ell=2$.
By trace theorem,
$V_{p,\ell}=W^{1,p}(\Omega)$ if $ 1\leq\ell< p(n-1)/(n-p)$. Otherwise,
$V_{p,\ell}\subset_{\not=}W^{1,p}(\Omega)$.

We formulate the problem under study as follows:

\noindent ($\mathcal P$)
 Find the temperature-potential pair
$(\theta,\phi)$ such that  if it verifies the variational problem:
\begin{align*}
\int_\Omega(k(\cdot,\theta)\nabla\theta)\cdot \nabla v\mathrm{dx}
+\int_{\Gamma}f_\lambda (\cdot,\theta)|\theta|^{\ell-2}
\theta v\mathrm{ds}=
\nonumber\\ =\int_\Omega \sigma(\cdot,\theta)\Big(
\alpha_{\rm s}(\cdot,\theta)(\theta+\phi)\nabla\theta+\phi
\nabla\phi\Big) \cdot\nabla v\mathrm{dx}
+\int_{\Gamma}\gamma(\cdot,\theta) \theta_\mathrm{e}^{\ell-1} 
v\mathrm{ds}; 
\\
\int_\Omega(\sigma(\cdot,\theta)\nabla\phi)\cdot \nabla w\mathrm{dx}=
-\int_\Omega
\left({ \sigma(\cdot,\theta)}\alpha_{\rm s}(\cdot,\theta)\nabla\theta\right)
\cdot \nabla w\mathrm{dx}+
\int_{\Gamma_{\rm N}}g w\mathrm{ds},
\end{align*}
for every $v\in V_{p',\ell}$ and $w\in V_{p'}$, where 
 $p'$ accounts for the conjugate exponent of $p$: $p'=p/(p-1)$.

We emphasize that this solution verifies, in the distributional sense,
the PDE written in terms of the Joule and Thomson effects:
\[
0=\nabla\cdot{\bf q}+\nabla\phi\cdot{\bf j}
=-\nabla\cdot(k_{\rm T}\nabla\theta)-{|{\bf j}|^2\over \sigma}
+\mu\nabla\theta\cdot{\bf j},
\]
where the
 Thomson coefficient $\mu$ is the  thermoelectric coefficient
directly measurable for individual materials that satisfies the second Kelvin 
relation: $\mu(T)=T{\partial\alpha_{\rm s}\over \partial T}(T).$
The verification of the second law of thermodynamics is stated in Appendix.

We assume that

(H1) The  Seebeck coefficient $\alpha_{\rm s}
:\Omega\times\mathbb{R}\rightarrow\mathbb{R}$ is a
  Carath\'eodory function, 
{\em i.e.} measurable with respect to $x\in\Omega$ and
  continuous with respect to $T\in\mathbb R$, such that
   \begin{equation}
\exists\alpha^\#>0:\quad
|\alpha_{\rm s}(x,T)|\leq \alpha^\#,\quad\mbox{a.e. } x\in \Omega,
\quad\forall T\in \mathbb R.\label{amm}
  \end{equation}
  
(H2) The thermal and electrical conductivities 
 $k,\sigma:\Omega\times\mathbb{R}\rightarrow\mathbb{M}_{n\times n}$ are
  Carath\'eodory tensors, where $\mathbb{M}_{n\times n}$ denotes the set
  of $n\times n$ matrices. Furthermore, they verify
 \begin{align*}
\exists k_\#>0:\quad k_{ij}(x,T)\xi_i\xi_j\geq k_\#|\xi|^2;&\\
\exists\sigma_\#>0:\quad \sigma_{ij}(x,T)\xi_i\xi_j\geq 
\sigma_\#|\xi|^2,&\quad\mbox{a.e. } x\in \Omega,\quad\forall T\in \mathbb{R},
\ \xi\in\mathbb{R}^n,
  \end{align*}
under the summation convention over repeated indices:
$\mathsf{A}{\bf a}\cdot{\bf b}=A_{ij}a_jb_i={\bf b}^\top  \mathsf{A}{\bf a}$;
and
 \begin{align}
\exists k^\#>0:\quad | k_{ij}(x,T)|\leq k^\#;&\nonumber\\
\exists\sigma^\#>0:\quad |\sigma_{ij}(x,T)|\leq 
\sigma^\#,&\quad\mbox{a.e. } x\in \Omega,\quad\forall T\in \mathbb R,
\label{smm}
  \end{align}
  for all $i,j\in\{1,\cdots,n\}$.
  
  (H3)  The boundary operators $f_\lambda$ and $\gamma$ are
 Carath\'eodory  functions from $\Gamma\times\mathbb{R}$
 into $\mathbb{R}$ such that
 \begin{align}
\exists b_\#, b^\#>0:\quad b_\#\leq f_\lambda(x,T)\leq 
b^\#;&\label{fmm}\\
\exists \gamma^\#>0:\quad |\gamma (x,T)|\leq 
\gamma^\#,&\quad\mbox{a.e. } x\in \Gamma,\quad\forall T\in \mathbb R.
\label{gmm}
  \end{align}

(H4)   $\theta_\mathrm{e}\in L^{(\ell-1)(2+\delta)}(\Gamma)$, and
$g\in L^{2+\delta}(\Gamma_{\rm N})$ such that $\int_{\Gamma_{\rm N}}g
\mathrm{ds}=0$, for some $\delta>0$. For the sake of simplicity, we assume
 $\delta=1$.

Observe that for any $1\leq p\leq 3$,
$ L^{3}(\Gamma_{\rm N})\hookrightarrow L^{p(n-1)/n}(\Gamma_{\rm N})$
 (which is the dual space of  $ L^{p'(n-1)/(n-p')}(\Gamma_{\rm N})$),
then $gw\in L^1(\Gamma_{\rm N})$ for all  $w\in W^{1,p'}(\Omega)$.

Finally,  we are able to state
the  existence result  in the two-dimensional space.
Let us denote by $C_\infty
=n^{-1/p}\omega_n^{-1/n}[(p-1)/(p-n)]^{1/p'}|\Omega |^{1/n-1/p}$ 
the continuity constant of  the Morrey-Sobolev embedding
  $W^{1,p}(\Omega)\hookrightarrow L^\infty(\Omega)$ for $p>n$ \cite{tal94},
  where 
$\omega_n$ is the volume of the  unit ball $B_1(0)$ of
$\mathbb{R}^n$, that is, $\omega_n=\pi^{n/2}/\Gamma(n/2+1)$.
\begin{theorem}\label{main2}
Suppose that the assumptions (H1)-(H4) be fulfilled, under $n=2$.
Then,  there exists at least one  
solution $(\theta,\phi)\in V_{p,\ell}\times V_{p}$ of (${\mathcal P}$), 
for $2<p<2+1/(\upsilon-1)$, where 
\begin{equation}\label{defu}
\upsilon=
65\times 2^{12}\left( 6\sqrt{2}S_{1}
\max\{{\sqrt{4(\sigma^\#)^2+ \sigma_\#}\over \sigma_\#},
{\sqrt{4(k^\#)^2+k_\#}\over k_\#}\}
+1\right)^2,
\end{equation}
with $S_{1}$ being a Sobolev continuity constant
 (see Remark \ref{rsob}),
if provided by the data smallness conditions $\|g\|_{p,\Gamma_N}<1$ and
 $\mathcal{Q}(1)<1$,
where $  \mathcal{Q}$ is given in (\ref{defQ}).
\end{theorem}

\begin{remark}\label{rsob}
For $1<q<n$,  the best continuity constant of 
 the Sobolev embedding $  {W}^{1,q}(\Omega)
 \hookrightarrow { L}^{q^*}(\Omega)$,
with   $q^*=qn/(n-q)$ being the critical Sobolev exponent, is 
(for smooth functions that decay at infinity, \cite{tale})
\[
S_q=\pi^{-1/2}n^{-1/q}\left({q-1\over n-q}\right)^{1-1/q}\left[
{\Gamma(1+n/2)\Gamma(n)\over \Gamma (n/q)\Gamma(1+n-n/q)}\right]^{1/n}.
\]
For $1^*=n/(n-1)$,  there exists the limit  constant
$S_1=\pi^{-1/2}n^{-1}[\Gamma(1+n/2)]^{1/n}$ \cite{tale}.
\end{remark}

\section{Abstract main results}
\label{state}

Let $\Omega\subset \mathbb{R}^n$ ($n\geq 2$) be a domain of class $C^{1}$
with the following characteristics.
\begin{definition}\label{domega}
Its boundary $\partial\Omega$ is constituted by two disjoint
 open $(n-1)$-dimensional sets,
$\Gamma_{\rm N}$ and $\Gamma$,
such that $\partial\Omega=\bar\Gamma_{\rm N}\cup \bar\Gamma$.
\end{definition}
We consider $\Gamma_{\rm N}$ over which the Neumann 
boundary condition is taken into account,
and $\Gamma$ over which the radiative effects may occur.
Each one, $\Gamma_{\rm N}$ and $\Gamma$, may be alternatively of zero $(n-1)$-Lebesgue
measure.

We study the following boundary value problem, in the sense of distributions,
\begin{align}
-\nabla\cdot(   \mathsf{A}
\nabla u)=f-\nabla\cdot{\bf f}&\mbox{ in }\Omega;\label{omega}\\
(\mathsf{A}\nabla u-{\bf f})\cdot{\bf n}+b(u)=h&\mbox{ on }\Gamma;
\label{robin}\\
(\mathsf{A}\nabla u-{\bf f})\cdot{\bf n}=g&\mbox{ on }\Gamma_{\rm N}, \label{gama}
\end{align}
where $\bf n$ is the unit outward normal to the boundary $\partial\Omega$.
Whenever the $(n\times n)$-matrix of the leading coefficient
is $\mathsf{A}=aI$, where $a$ is a real function and $I$ denotes
the identity matrix,  the elliptic  equation stands for isotropic materials.
Our problem includes the conormal derivative boundary value problem if 
provided by $\Gamma=\partial\Omega$
(or equivalently $\Gamma_{\rm N}=\emptyset$). The problem (\ref{omega})-(\ref{gama})
is the so-called mixed 
Robin-Neumann problem if $b$ is linear in  (\ref{robin}).
 
Assume
\begin{description}
\item[(A)]
 $\mathsf{A}=[A_{ij}]_{i,j=1,\cdots,n}
\in [L^\infty(\Omega)]^{n\times n}$ is uniformly elliptic, and uniformly bounded:
\begin{align}\label{amin}
\exists  a_\#>0,&\quad
A_{ij}(x)\xi_i\xi_j\geq a_\#|\xi|^2,
\quad\mbox{ a.e. }x\in\Omega,\ \forall \xi\in\mathbb{R}^n;\\
\exists  a^\#>0,&\quad \|\mathsf{A}\|_{\infty,\Omega}\leq a^\#.\label{amax}
\end{align}
\item[(B)]
 $b:\Gamma\times \mathbb{R}\rightarrow \mathbb{R}$ is a Carath\'eodory function
such that it is monotone with respect to the last variable, and it has 
 $(\ell-1)$-growthness properties:
\begin{align}\label{bmin}
\exists  b_\# >0, &\quad b(x,T){\rm sign}(T)\geq  b_\#|T|^{\ell-1};\\
\exists  b^\#>0, &\quad |b(x,T)|\leq b^\#|T|^{\ell-1},\label{bmax}
\end{align} 
for a.e. $x\in\Gamma$, and for all $T\in \mathbb{R}$.
\end{description}

\begin{remark}\label{rmo}
If $b(T)=|T|^{\ell-2}T$, for all $T\in\mathbb{R}$, the property of
strong monotonicity occurs with $b_\#=2^{(2-\ell)}$ \cite[Lemma 3.3]{dpz}.
\end{remark}

Our main abstract results are stated as follows.
We observe that (\ref{cotam1}) has
 no a standard format in order to avoid inflated involved constants.
\begin{theorem}[meas$(\Gamma)>0$]\label{main1}
Let $\delta>0$. 
Let   ${\bf f}\in {\bf L}^{2+\delta}(\Omega)$,  $f\in L^{2+\delta}(\Omega)$,
$g\in L^{2+\delta}(\Gamma_{\rm N})$, 
 and $h\in L^{2+\delta}(\Gamma)$.
Under the assumptions (A)-(B), there exists  a weak solution 
$u \in V_{2,\ell}$ to (\ref{omega})-(\ref{gama}), in the sense
\begin{align}
\int_{\Omega}    ( \mathsf{A}\nabla u)\cdot
\nabla v \mathrm{dx}+\int_{\Gamma}b(u) v \mathrm{ds}
=\int_{\Omega}{\bf f}\cdot\nabla v \mathrm{dx}
+\int_{\Omega}fv \mathrm{dx}+\nonumber\\
+\int_{\Gamma_{\rm N}}gv \mathrm{ds}+\int_{\Gamma}hv \mathrm{ds},
 \quad\forall v\in V_{2,\ell},\label{pbu}
\end{align}
such that belongs to $W^{1,2+\varepsilon}(\Omega)$
for any $\varepsilon\in [0,\delta]\cap [0, 4 ((n+2)(\upsilon-1)) ^{-1} [$, where
$\upsilon=\upsilon_{\rm U}(a_\#,a^\#)$ is given by (\ref{defuu}).
In particular, if $\mathcal{M}={\rm  ess } \sup_\Omega |u| $ then
\begin{align}
\| \nabla u\|_{2+\varepsilon,\Omega}^{2+\varepsilon}
\leq (2^n+1)\left[
\left(8\over (r_\#)^n\right)^{\varepsilon/2} 
Z_1(\upsilon)
\| \nabla u\|_{2, \Omega}^{2+\varepsilon} +\right.\nonumber\\
+\left( 2^{(n+1)\varepsilon/2}Z_1(\upsilon)
+Z_2(\upsilon) \right)
\| {\mathcal F }(a_\#)\|_{2+\varepsilon, \Omega}^{2+\varepsilon}
+\nonumber\\
\left.+
\left( Z_1(\upsilon)+ Z_2(\upsilon)\right)
 \| {\mathcal H}(a_\#, b^\#)\|_{2+\varepsilon,\partial\Omega}^{2+\varepsilon}\right]
,\quad\label{cotam1}
\end{align}
where 
 \begin{align}
 {\mathcal H}(a_\#, b^\#) &={2K_{{2n/( n+1)}}
\over (a_\#)^{1/2}}
\left({2\over  a_\#}+{ 2^{-1/n}}\right)^{1/2} 
|g\chi_{\Gamma_{\rm N}}+\left(h+b^\#\mathcal{M}^{\ell-1}
\right) \chi_{\Gamma}|; \nonumber \\
\label{fa}
{\mathcal F}(a_\#) &= (a_\#) ^{-1/2}\left[\left({2\over  a_\#}+ 2
\right)|{\bf f}|^{2} 
+{ 1\over \nu_3}|f|^{2} \right]^{1/2};\\
\label{zz1}
Z_1(\upsilon)&= {4
\over 4-(n+2)(\upsilon-1)\varepsilon}\times 2^{n(1+\varepsilon /2)};\\
\label{zz2}
Z_2(\upsilon)&=
{ \upsilon(4+(n+2)\varepsilon)
\over 4-(n+2)(\upsilon-1)\varepsilon}\times 2^{n(1+\varepsilon/2)},
\end{align} 
with $ K_{2n/(n+1)}=
[\Gamma(n)]^{1/(n+1)}[ (\sqrt{\pi}n)^{n-1}\Gamma((n+1)/2)]^{-1/(n+1)}$
(see Remark \ref{rk}), and $r_\#$ according to (\ref{rstm}).
\end{theorem}

The solvability of the boundary-value problem (\ref{omega})-(\ref{gama})
and some properties of the weak solutions can be found in \cite{lc-ijpde,sinica}.
In particular, for $f=g=0$ we have the following $L^\infty$-estimate ($p>n\geq 2$)
\begin{equation}\label{supess}
{\rm  ess } \sup_{\Omega\cup \partial\Omega}|u|\leq 1
+
\mathcal{Z}_1(a_\#,b_\#)\|{\bf f}\|_{p,\Omega} + 
\mathcal{Z}_2(a_\#,b_\#)\|h\|_{p,\Gamma} ,
\end{equation}
with
\begin{align*}
 \mathcal{Z}_1(a_\#,b_\#)&=
 \left({|\Omega|^{1/(2\alpha)}\over a_\#}
 +{1\over \sqrt{a_\#b_\#}}\right)
|\Omega|^{\frac{p-2}{4p}} \mathcal{Z} ;\\
 \mathcal{Z}_2(a_\#,b_\#) &= \left({1\over b_\#}
+{|\Omega|^{1/( 2\alpha)}\over  \sqrt{a_\#b_\#}}\right) \mathcal{Z} ;\\
\mathcal{Z} &=
2^{\frac{\alpha(p-2)+2p}{ \alpha(p-2)-2p}}
(|\Omega|+|\partial\Omega|)^{\frac{\alpha(p-2)-2p}{4p\alpha}}
( S_{2\alpha/(\alpha+2)}+K_{2\alpha /(\alpha+2) }),
\end{align*}
where  $\alpha>2p/(p-2)$.

\begin{remark}\label{rk}
For $1<q<n$, the best continuity constant of the embedding
$   W^{1,q}(\Omega)\hookrightarrow L^{q_*}(\partial\Omega)$,
 with $q_*=q(n-1)/(n-q)$ being the critical trace exponent,
is \cite{bond}
\[
K_q=\pi^{(1-q)/2}\left({q-1\over n-q}\right)^{q-1}\left[
{\Gamma\left({q(n-1)\over 2(q-1)}\right)\Big/ \Gamma \left({
n-1\over 2(q-1)}\right)}\right]^{(q-1)/(n-1)},\]
where $\Gamma$ stands for the Gamma function.
\end{remark}

\begin{theorem}[meas$(\Gamma)=0$]\label{main1n}
Let the assumption (A) be fulfilled, and $\delta>0$. 
If  ${\bf f}\in {\bf L}^{2+\delta}(\Omega)$,  $f\in L^{2+\delta}(\Omega)$,
and $g\in L^{2+\delta}(\Gamma_{\rm N})$, verify the compatibility condition
\begin{equation}\label{cc}
 \int_{\Omega}f\mathrm{dx}
+\int_{\Gamma_{\rm N}}g \mathrm{ds}=0,
\end{equation}
then the Neumann problem
\begin{equation}
\int_{\Omega}    ( \mathsf{A}\nabla u)\cdot
\nabla v \mathrm{dx}
=\int_{\Omega}{\bf f}\cdot\nabla v \mathrm{dx}
+\int_{\Omega}fv \mathrm{dx}
+\int_{\Gamma_{\rm N}}gv \mathrm{ds},
 \quad\forall v\in V_{p'},\label{pbun}
\end{equation}
admits a unique weak solution  $u\in V_{p}$ satisfying
(\ref{cotam1}) with  $ {\mathcal H}$ being replaced by
\begin{equation}\label{gan}
 {\mathcal G}(a_\#)= 2{K_{{2n/( n+1)}}
\over (a_\#)^{1/2}}
\left({2\over a_\#}+{ 2^{-1/n}}\right)^{1/2} 
|g|
 .\end{equation}
\end{theorem}

Throughout this work, we adopt the standard notations:
 
- $Q_R(x)$ denotes the cubic interval 
(cubes with edges parallel to coordinate planes), that is,
 $Q_R(x)$ denotes the open ball of  radius $R>0$ centered at
 the point $x\in\mathbb{R}^n$, defined by
 \[Q_R(x)=\{y\in\mathbb{R}^n:\ |y-z|:=\max_{1\leq i\leq n}|y_i-x_i|<R\}.
 \]
 We call by $Q$ any cube that is an orthogonal transformation of a cubic 
 interval.

 -  $A[v>k]=\{x\in A:\ v(x)>k\}$, where $v\in L^1(A)$,
$v\geq 0$ in $A$, with the set $A$ being either
   $\Omega$, $\Gamma_{\rm N}$, $\Gamma$, $\partial\Omega$ or $\bar \Omega$.
Moreover, the significance of $|A|$ stands  the Lebesgue measure of
 a set of $\mathbb{R}^n$, and also
for the $(n-1)$-Lebesgue measure.

\section{Reverse H\"older inequalities with increasing supports}
\label{reverse}

In this section, a $C^{0,1}$ domain is sufficient to be assumed.

Let us recall a result on the Stieltjes integral in the form that
we are going to use (for the general form see \cite{ark}).
\begin{lemma}\label{hh}
Suppose that $q,t_0\in ]0,\infty[$, and $a\in ]1,\infty[$.
If $h,H:[t_0,\infty[\rightarrow [0,\infty[$ are nonincreasing functions such that
\begin{equation}\label{lim0}
\lim_{t\rightarrow\infty}h(t)=
\lim_{t\rightarrow\infty}H(t)=0,
\end{equation}
and that
\begin{equation}\label{hhq}
-\int_t^\infty \tau^q \mathrm{dh(\tau)}\leq a[t^qh(t)+H(t)],\quad
\forall t\geq t_0,
\end{equation}
then, for $\gamma\in [q,aq/(a-1)[$
\begin{align}\label{hhtese}
-\int_{t_0}^\infty t^\gamma \mathrm{dh(t)}\leq {q\over aq-(a-1)\gamma}
\left(-\int_{t_0}^\infty t^q \mathrm{dh(t)}\right) +\nonumber\\
+{a\gamma\over aq-(a-1)\gamma}
\left(-\int_{t_0}^\infty t^{\gamma-q} \mathrm{dH(t)}\right).
\end{align}
\end{lemma}

Next, let us recall the Calderon-Zygmund subdivision argument
\cite[p. 127]{gia83}.
\begin{lemma}\label{caldz}
Let $Q$ be a open cube in $\mathbb{R}^n$,
$v\in L^1(Q)$, $v\geq 0$ in $Q$, and $\zeta> (v)_Q=
{1\over |Q|}\int_{Q}v(y)\mathrm{dy}$.
Then there exists a sequence of  cubes $\{Q_j\}_{j\geq 1}$,
with sides parallel to the axes and with disjoint interiors, such that
$v\leq\zeta$ a.e. in $ Q\setminus \left(\cup_{j\geq 1}Q_j\right)$, and
\[
\zeta<
{1\over |Q_j|}\int_{Q_j}v(y)\mathrm{dy}\leq 2^n\zeta.\]
\end{lemma}

Now we prove the following
versions with increasing supports
 of Gehring lemma. Their proofs are based on classical arguments
 \cite{arklady,gia83,stre}.
\begin{proposition}\label{gia-geh}
Let $p>1$, $\delta>0$, and nonnegative functions
$\Phi\in L^p(\Omega)$ and $\Psi\in L^{p+\delta}(\Omega)$
satisfy the estimate
\begin{equation}\label{hisum}
\frac 1{R^n}\int_{Q_{\alpha R}(z)} \Phi^p\mathrm{dx}
\leq B
\Big(\frac 1{R^n}\int_{Q_R(z)} \Phi \mathrm{dx}\Big)^{p}
+\frac 1{R^n}\int_{Q_R(z)} \Psi^p\mathrm{dx},
\end{equation}
for all $z\in\Omega$,
 $R< \min\{\mbox{\rm dist}(z,\partial\Omega)/\sqrt{n},R_0\}$
with some constants $\alpha\in [1/2,1[$, $R_0>0$,  and  $B>0$.
Then,
$\Phi\in L^{q}_{\rm loc}(\Omega)$ for all $p\leq q\leq p+\delta$ and 
$q<p+(p-1)/(a_{\rm I}-1)$,  with
 \begin{equation}\label{defva}
 \varkappa = (8^n+1) 2^{3np}\left( B^{1/p}+1\right)^p
.\end{equation}
 In  particular, for any cubic interval
 $Q_r(x_0)\subset \subset\Omega$ such that $r<R_0$, we have
\begin{align}
\| \Phi\|_{q,\omega}^q
\leq \left[ 
 {\rm dist}(\omega,\partial Q_r(x_0))\right]^{-nq/p}
\left[ {\varkappa (q-1)r^{nq/p}\over q-1-\varkappa (q-p)}
\| \Psi\|_{q,Q_r(x_0)}^q+
\right.\nonumber\\
\left.+ { (p-1)r^n\over q-1-\varkappa (q-p)}
\left( \| \Phi\|_{p,Q_r(x_0)}^p+
\| \Psi\|_{p,Q_r(x_0)}^p\right)^{q/p}\right],
\label{higher}
\end{align}
for any measurable set $\omega\subset\subset Q_r(x_0)$.
\end{proposition}
\begin{proof}
Fix $Q_r(x_0)\subset\subset\Omega$ with $r<R_0$.
Let us transform the cubic interval
 $Q_r(x_0)$ into $Q=Q_{3/2}(0)$ by the passage to
new coordinates system $y=3(x-x_0)/(2r)$. 
Setting $M=[3/(2r)]^{n/p}(\|\Phi\|_{p,Q_r(x_0)}^p+\|\Psi\|_{p,Q_r(x_0)}^p)^{1/p}$, 
the normalized functions
$\overline\Phi(y)=\Phi(x_0+2ry/3)/M$ and $\overline\Psi(y)=\Psi(x_0+2ry/3)/M$ 
satisfy $\max\{\|\overline\Phi\|_{p,Q},\|\overline\Psi\|_{p,Q}\}\leq 1$.
Let us define  $\Phi_0(y)= 
\overline\Phi(y){\rm dist}^{n/p}(y,\partial Q)$, and 
$\Psi_0 (y)= \overline\Psi(y){\rm dist}^{n/p}(y,\partial Q)$.

For each $t\in [1,\infty[$, we introduce
\begin{align*}
h(t)=\int_{Q[\Phi_0 >t]}\Phi_0(y)\mathrm{dy},\\
H(t)=\int_{Q[\Psi_0 >t]}\Psi_0^p(y)\mathrm{dy}.
\end{align*}
Then,
 $h,H:[1,\infty[\rightarrow [0,\infty[$ are nonincreasing functions such that
verify (\ref{lim0}). In order to apply Lemma \ref{hh}, it remains to prove that
(\ref{hhq}) is verified with $q=p-1$,  
taking the relation 
\[
\int_{Q[\Phi_0>t]}\Phi_0^p (y)\mathrm{dy}=-\int_t^\infty\tau^{p-1}
\mathrm{dh(\tau)},\quad\forall p>1,
\]
into account. More exactly,  we must prove that
\begin{equation}\label{phipsi}
\int_{Q[\Phi_0 >t]}\Phi_0^p(y)\mathrm{dy}\leq \varkappa  \left(t^{p-1}
\int_{Q[\Phi_0 >t]}\Phi_0(y)\mathrm{dy}+
\int_{Q[\Psi_0 >t]}\Psi_0^p(y)\mathrm{dy}\right),
\end{equation}
for any $t\geq 1$.

We decompose $Q=\cup_{k\in\mathbb{N}_0}C^{(k)}$,
where  $C^{(0)}=Q_{1/2}(0)$, and for each $k\geq 1$,
$C^{(k)}=\{y\in Q:\ 2^{-k}<{\rm dist}(y,\partial Q)\leq 2^{-k+1}\}$.
Each set $C^{(k)}$ is the finite union of disjoint cubic intervals 
of size $1/2^{k+2}$, namely $D^{(k)}_i=
Q_{1/2^{k+3}}(w^{(i)})$.  In particular, $|D^{(k)}_i|=2^{-(k+2)n}$.

Fix $t \geq 1$.
Since we  have
\[
{1\over |D^{(k)}_i|}\int_{D^{(k)}_i}
\Phi_0^p(y)\mathrm{dy}\leq 2^{3n}\|\overline\Phi\|_{p,Q}^p
\leq 2^{3n},
\]
from Lemma \ref{caldz} with $v=\Phi^p_0\in L^1(D^{(k)}_i)$
 and $\zeta=t^p{\lambda}$ with $\lambda>2^{3n}$ defined in (\ref{lambd}),
 there exists a disjoint sequence of
cubic intervals $Q^{(k)}_{i,j}\subset D^{(k)}_i$
in the conditions of Lemma. 
Since $i\in\{1,\cdots,I\}$ with $I\in\mathbb{N}$, we in fact have
a disjoint sequence of 
cubic intervals $Q^{(k)}_{j}=Q_{r_j^{(k)}}(y^{(k,j)})\subset C^{(k)}$
 such that $\Phi_0\leq t\sqrt[p]{\lambda}$ a.e. in 
$C^{(k)}\setminus \left(\cup_{j\geq 1} Q^{(k)}_j\right):=E$, and
\begin{equation}
t^p\lambda<
{1\over |Q^{(k)}_j|}\int_{Q^{(k)}_j}\Phi_0^p(y)\mathrm{dy}\leq 2^n
t^p\lambda,
\quad \forall j\geq 1.\label{sigma}
\end{equation} 

Considering that $|E[\Phi_0>t\sqrt[p]{\lambda}]|=0$, we compute
\begin{equation}\label{sum}
\int_{Q[\Phi_0 >t\sqrt[p]{\lambda}]}\Phi_0^p\mathrm{dy}\leq
\sum_{k\geq 0}
\sum_{j\geq 1}\int_{Q^{(k)}_j}\Phi_0^p\mathrm{dy}\leq 2^n t^p\lambda
\sum_{k\geq 0}\sum_{j\geq 1}|Q^{(k)}_j|.
\end{equation}

Next, to estimate the above right hand side,
let us prove, for all $k\geq 0$, and $j\geq 1$, there exists
 $R=R_{kj}\in ]r^{(k)}_j,2r_j^{(k)}]$ that verifies
\begin{equation}\label{y0}
t(2R)^n< \int_{Q_R[\Phi_0>t]}\Phi_0\mathrm{dy}+t^{-p+1}
\int_{Q_R [\Psi_0>t]}\Psi_0^p\mathrm{dy},
\end{equation}
with the notation $Q_R= Q_{R}(y^{(k,j)})$.
Since  $R\leq 2r_j^{(k)}<2^{-(k+1)}$, 
$Q_{R}$ only intersects the sets $C^{(k-1)}$, $C^{(k)}$,
and $C^{(k+1)}$.

Fix $Q^{(k)}_j\subset \subset C^{(k)}$  such that
$r^{(k)}_j< 2^{-(k+3)}$.
Let us choose $R\in ]r^{(k)}_j, 2r^{(k)}_j]$ provided by
 $\alpha= r^{(k)}_j/ R \in [1/2,1[$.
 We use the first inequality in (\ref{sigma}), obtaining
\[
t^p\lambda<
{1\over |Q^{(k)}_j|}\int_{Q^{(k)}_j}\Phi_0^p\mathrm{dy}\leq
{1\over R^n}\int_{Q_{\alpha R}}\Phi_0^p\mathrm{dy}\leq
{2^{-(k-1)n}\over R^n}\int_{Q_{\alpha R}}\overline\Phi^p\mathrm{dy},
\] 
with $ Q_{\alpha R}= Q_{\alpha R}(y^{(k,j)})\subset C^{(k)}$.

Rewriting (\ref{hisum})  in terms of the new coordinates system,
taking $z=x_0+2ry^{(k,j)}/3$, and dividing the resultant inequality by $\|\Phi\|_{p,Q_r(x_0)}^{p}+\|\Psi\|_{p,Q_r(x_0)}^{p}$, we deduce
\[{1\over R^n}
\int_{Q_{\alpha R}}\overline\Phi^p\mathrm{dy}\leq
B \left({1\over R^n}\int_{Q_{R}}\overline\Phi\mathrm{dy}
\right)^p+{1\over R^n}
\int_{Q_{R}}\overline\Psi^p\mathrm{dy},
\]
considering that $r<R_0$.

Then, gathering the above two inequalities, we find
\begin{equation}\label{rr}
(tR^n)^p\lambda< 2^{-(k-1)n}
\left[B\left(\int_{Q_{R}}\overline\Phi\mathrm{dy}
\right)^p+R^{n(p-1)}
\int_{Q_{R}}\overline\Psi^p\mathrm{dy}\right].
\end{equation}

On one hand,  we have
\[
(tR^n)^p\lambda< 2^{n}\left[B
\left(\int_{Q_{R}}
\Phi_0\mathrm{dy}\right)^p+R^{n(p-1)}
\int_{Q_{R}}\Psi_0^p\mathrm{dy}\right],
\]
obtaining
\begin{align*}
tR^n\sqrt[p]{\lambda}< 2^{n/p}\left[
B^{1/p}\left(\int_{Q_R [\Phi_0>t]}
\Phi_0\mathrm{dy}+t(2R)^n\right)+
\right.\\+\left.
R^{n(p-1)/p}\left(\int_{Q_R [\Psi_0>t]}
\Psi_0^p\mathrm{dy}\right)^{1/p}+t2^{n/p}R^n\right].
\end{align*}
By applying the Young inequality,
\[
R^{n(p-1)/ p}\left(\int_{Q_R [\Psi_0>t]}
\Psi_0^p\mathrm{dy}\right)^{1/p}\leq
{t\over p'}R^{n}+
{t^{-(p-1)}\over p}\int_{Q_{R} [\Psi_0>t]}
\Psi_0^p\mathrm{dy},
\]
and taking $p>1$ and $2^{n/p}+1/p'\leq 2^n$ into account, we find
\begin{align*}
tR^n\left[\sqrt[p]{\lambda}- 2^{n+n/p}\left(
B^{1/p}
+1\right)\right]<\\ < 2^{n/p}\left[
B^{1/p}\int_{Q_R[\Phi_0>t]}
\Phi_0\mathrm{dy}+t^{-p+1}
\int_{Q_{R} [\Psi_0>t]}\Psi_0^p\mathrm{dy}\right].
\end{align*}
Therefore,  we choose
\begin{equation}\label{lambd}
\lambda= 2^{3np}\left( B^{1/p}+1\right)^p, 
\end{equation}
 concluding (\ref{y0}).

On the other hand, applying the H\"older inequality in (\ref{rr}), and using  $\max\{\|\overline\Phi\|_{p,Q},$ $\|\overline\Psi\|_{p,Q}\}\leq 1$,
 we have
\[
tR^n< \lambda^{-1/p}
\left(B^{1/p}+1\right) R^{n(p-1)/p},
\]
and consequently, using (\ref{lambd}), we conclude
\[
tR^{n/p}<2^{-3n}.
\]

According to the Vitali covering lemma, there exist
$\sigma\in ]3,4[$ and a sequence of disjoint cubic intervals
$\{Q_{R_i}(y^{(i)})\}_{i\geq 1}$ from the collection
$\{Q_{R}(y^{(k,j)})\}_{k\geq 0,\, j\geq 1}$  such that
\[
\cup_{k\geq 0}\cup_{j\geq 1}Q_R(y^{(k,j)})\subset
\cup_{ i\geq 1}Q_{\sigma R_i}(y^{(i)})
\subset Q.\]
Hence,\[
\sum_{k\geq 0}\sum_{j\geq 1}|Q^{(k)}_j|\leq
\sum_{k\geq 0}\sum_{j\geq 1}
|Q_R(y^{(k,j)})|\leq\sigma^n \sum_{i\geq 1}|Q_{R_i}(y^{(i)})|.
\]
Combining the above with (\ref{sum}), and (\ref{y0}), we find
\begin{equation}\label{sums}
\int_{Q[\Phi_0 >t\sqrt[p]{\lambda}]}\Phi_0^p\mathrm{dy}\leq
 2^n \lambda\sigma^n \left(t^{p-1}\int_{Q[\Phi_0>t]}
\Phi_0\mathrm{dy}+
\int_{Q[\Psi_0>t]}\Psi_0^p\mathrm{dy}\right).
\end{equation}

Now, observing that
\[
\int_{Q[\Phi_0>t]}\Phi_0^p\mathrm{dy}\leq
\int_{Q[\Phi_0>t\sqrt[p]{\lambda}]}\Phi_0^p\mathrm{dy}+t^{p-1}\lambda
\int_{Q[\Phi_0>t]}\Phi_0\mathrm{dy},\]
 (\ref{sums}) implies (\ref{phipsi}).
Therefore, Lemma \ref{hh} can be applied, concluding that, for any $\gamma$
such that $p\leq \gamma+1<p+{(p-1)/( \varkappa -1)}$, (\ref{hhtese}) 
implies
\begin{align*}
\int_{Q[\Phi _0>1]}\Phi_0^{\gamma+1}\mathrm{dy}\leq
{p-1\over \varkappa (p-1)-(\varkappa -1)\gamma} 
\int_{Q[\Phi_0>1]}\Phi_0^p\mathrm{dy}+\\
+{\varkappa \gamma\over \varkappa (p-1)-(\varkappa -1)\gamma}
\int_{Q[\Phi_0>1]}
\Psi_0^{\gamma+1}\mathrm{dy}.
\end{align*}
The requirement of $q=\gamma +1<p+\delta$ assures the finiteness of the last
integral of the RHS of the above inequality.
Since $\Phi_0^{\gamma+1}\leq\Phi_0^{p}$ a.e. in $Q\setminus Q[\Phi_0>1]$,
for any $\omega\subset\subset Q$, we find
\begin{align*}
\left[{3{\rm dist}(\omega,\partial Q_r(x_0))\over 2r}\right]^{nq/ p}
\int_{\omega}\overline\Phi^{q}\mathrm{dy}\leq 
{p-1\over q-1-\varkappa (q-1)} \left({3\over 2}\right)^n
\int_{Q}\overline\Phi^p\mathrm{dy}+\\
+{\varkappa (q-1)\over q-1-\varkappa (q-1)}
\left({3\over 2}\right)^{nq/p}
\int_{Q}\overline\Psi^{q}\mathrm{dy},
\end{align*}
 keeping the same designation to the transformed
set $\omega\subset \subset  Q_r(x_0)$.
Passing the above inequality to the initial coordinates system,
 we conclude (\ref{higher}).
\end{proof}

\begin{corollary}\label{high}
Let $x_0\in\Omega$ and $R_0>0$ such that $\sqrt{n}R_0<
{\rm dist}(x_0,\partial\Omega)$.
Let $p>1$, $\delta>0$ and nonnegative functions
$\Phi\in L^p(\Omega)$ and $\Psi\in L^{p+\delta}(\Omega)$
satisfy the estimate
\begin{equation}\label{hisum2}
\frac 1{R^n}\int_{Q_{\alpha R}(z)} \Phi^p\mathrm{dx}
\leq B
\Big(\frac 1{R^n}\int_{Q_R(z)} \Phi^t \mathrm{dx}\Big)^{p/t}
+\frac 1{R^n}\int_{Q_R(z)} \Psi^p\mathrm{dx}, 
\end{equation}
for all  $z\in Q_{R_0}(x_0)$ and $\sqrt{n}R<
{\rm dist}(z,\partial  Q_{R_0}(x_0))$,
with some constants $\alpha\in [1/2,1[$,
 $t<p$,   and  $B>0$.
Then $\Phi\in L^{q}_{\rm loc}( Q_{R_0}(x_0))$ for all $q
\in [p,p+\delta]\cap [p,p+(p-t)/(\varkappa -1)[$, with $\varkappa $ being
defined by (\ref{defva}).
 In  particular, it verifies
\begin{align}
\| \Phi\|_{q,Q_{r/2}(x_0)}^q
\leq 
{2^{nq/p}\over q-t-\varkappa (q-p)}\left[ 
(p-t)r^{-n(q-p)/p}2^{(q-p)/p}\| \Phi\|_{p, Q_{r}(x_0)}^{q}
+\right.\nonumber\\
\left.+
\left(2^{(n+1)(q-p)/p}(p-t)+\varkappa 
(q-t)\right)\| \Psi\|_{q, Q_{r}(x_0)}^q
\right],\qquad
\label{highert}
\end{align}
for any $0<r< R_0$.
\end{corollary}
\begin{proof}
We write (\ref{hisum2}) as
\[
\frac 1{R^n}\int_{Q_{\alpha R}(z)} (\Phi^{t} )^{p/t}\mathrm{dx}
\leq B
\Big(\frac 1{R^n}\int_{Q_R(z)} \Phi^t \mathrm{dx}\Big)^{p/t}
+\frac 1{R^n}\int_{Q_R(z)}( \Psi^t)^{p/t}\mathrm{dx}, 
\]
and we apply Proposition \ref{gia-geh}
with $\omega=Q_{r/2}(x_0)$. Next, we use the H\"older inequality to get
\[\| \Psi\|_{p, Q_{r}(x_0)}^{p}\leq
(2r)^{n(q-p)/q}\| \Psi\|_{q, Q_{r}(x_0)}^{p}.\]
Then we apply the relation $(a^p+b^p)^{q/p}\leq 2^{q/p-1}(a^q+b^q)$
to conclude (\ref{highert}).
\end{proof}

Next,
the  higher summability of $\Phi$ near the flattened boundary is established
due to  the reverse H\"older inequality with a surface integral.
\begin{proposition}\label{surf}
Let $z_0=(z_0',0)\in\mathbb {R}^n$, $R_0>0$, and
\begin{align*}
Q^+_{R_0}(z_0):= \{z\in \mathbb R^n: |z'-z_0'|<R_0, \ z_n>0\};\\
\Sigma_{R_0}(z_0)=\{z\in \mathbb R^n:\ |z'-z_0'|<R_0, \ z_n=0\};\\
\partial'\Sigma_{R_0}(z_0)=\{z\in \mathbb R^n:\ |z'-z_0'|=R_0, \ z_n=0\},
\end{align*}
 where $|z'|=\max_{1\leq i\leq n-1}|z_i|$.
For $p>1$ and $\delta>0$, 
we suppose that the nonnegative functions
$\Phi\in L^p(Q^+_{R_0}(z_0))$, $\Psi\in L^{p+\delta}(Q^+_{R_0}(z_0))$, 
 and $\varphi\in L^{p+\delta}(\Sigma_{R_0}(z_0))$
satisfy the estimate
\begin{align}
\frac 1{R^n}\int_{Q^+_{\alpha R}(z)} \Phi^p\mathrm{dx}
\leq B
\Big(\frac 1{R^n}\int_{Q^+_R(z)} \Phi \mathrm{dx}\Big)^{p}
+\frac 1{R^n}\int_{Q^+_R(z)} \Psi^p\mathrm{dx}+\nonumber
\\
+\frac 1{R^{n-1}}\int_{\Sigma_R(z)} \varphi^p\mathrm{ds}, \label{hisums}
\end{align}
for all $z\in\Sigma_{R_0}(z_0)$,
 and all $R<\min\{R_0,{\rm dist}(z,\partial'\Sigma_{R_0}(z_0))\}$,
with some constants $\alpha\in [1/2,1[$,   and  $B>0$.
 Then, $\Phi\in  L^{p+\varepsilon}(\omega\cap Q_{R_0}^+(z_0))$,
for all $\varepsilon\in [0,\delta]\cap [0,(p-1)/(\varkappa -1)[$ and
measurable set $\omega\subset \subset Q_{R_0}(z_0)$,
and it verifies
\begin{align}
\| \Phi\|_{p+\varepsilon,\omega\cap Q_{R_0}^+(z_0)}^{p+\varepsilon}
\leq {
\left[{\rm dist}(\omega,\partial Q_{R_0}(z_0))\right]^{-n(1+\varepsilon/p)}\over p-1-(\varkappa -1)\varepsilon}
\times\nonumber\\
\times\left[
{(p-1)  R_0^n\over 2}
\left(2\| \Phi\|_{p, Q_{R_0}^+(z_0)}^p+ 2\| \Psi\|_{p, Q_{R_0}^+(z_0)}^p+
{2R_0\over 3} \| \varphi\|_{p, \Sigma_{R_0}(z_0)}^p\right)^{1+\varepsilon/p}
+\right.\nonumber\\
\left.+ \varkappa (p-1+\varepsilon) R_0^{n(1+\varepsilon/p)}
\left(
 \| \Psi\|_{p+\varepsilon, Q_{R_0}^+(z_0)}^{p+\varepsilon}+
\| \varphi\|_{p+\varepsilon, \Sigma_{R_0}(z_0)}^{p+\varepsilon}\right)
\right],
\qquad\label{highup}
\end{align}
where $\varkappa $ is given by (\ref{defva}).
\end{proposition}
\begin{proof}
 We prolong $\Phi$ and $\Psi$  as even functions with respect to 
 $\Sigma_{R_0}(z_0)$:
$$
\widetilde\Phi(z',z_n)=\Big\{
\begin{array}c
\Phi(z',z_n),\quad z_n>0 \\
\Phi(z',-z_n),\quad z_n<0
\end{array}
\qquad
\widetilde\Psi(z',z_n)=\Big\{
\begin{array}c
\Psi(z',z_n),\quad z_n>0 \\
\Psi(z',-z_n),\quad z_n<0.
\end{array}
$$

Transforming $ Q_{R_0}(z_0)$
 into $Q=Q_{3/2}(0)$ by the passage to
new coordinates system $y=3(z-z_0)/(2R_0)$,
setting $M=[3/(2R_0)]^{n/p}(\|\widetilde\Phi\|_{p,Q_{R_0}(z_0)}^p
+\|\widetilde\Psi\|_{p,Q_{R_0}(z_0)}^p+2R_0
\|\varphi\|_{p,\Sigma _{R_0}(z_0)}^p/3)^{1/p}$,  and defining  $\overline\Phi(y)=\widetilde\Phi(z_0+2R_0y/3)/M$,
 $\overline\Psi(y)=\widetilde\Psi(z_0+2R_0y/3)/M$, 
$\overline\varphi(y)=\varphi(z_0+2R_0y/3)/M$, we define
\begin{align*}
H(t)=\int_{Q[\overline\Psi >t]}\overline\Psi
^p(y)\mathrm{dy}+
\int_{\Sigma[\overline\varphi >t]}\overline\varphi^p(y)\mathrm{ds_y},\qquad
\Sigma=Q_{3/2}^{(n-1)}(0)\times \{0\}.
\end{align*}

The argument of the proof of Proposition \ref{gia-geh} remains valid
until (\ref{rr}), which reads
\[ 
(tR^n)^p\lambda< 2^{n} B
\left(\int_{Q_{R}}\Phi_0\mathrm{dy}
\right)^p+R^{n(p-1)}
\int_{Q_{R}}\overline\Psi^p\mathrm{dy}
+R^{n(p-1)+1} 
\int_{\Sigma_{R}}\overline\varphi^p\mathrm{ds_y}.
\] 
Evaluating the term
\[\left(R^{n(p-1)+1} 
\int_{\Sigma_{R}}\overline\varphi^p\mathrm{ds_y}\right)^{1/p}\leq
\left(R^{n(p-1)}
\int_{\Sigma_{R}[\overline\varphi>t]}\overline\varphi^p\mathrm{ds_y}\right)^{1/p}+t
2^{n-1\over p}R^{n},
\]
and proceeding as in  the proof of Proposition \ref{gia-geh},
 we obtain
\begin{align*}
tR^n\left[\sqrt[p]{\lambda}- 2^{n}\left(2 ^{n/p}
B^{1/p}+2\right)\right]<2^{n/ p}
B^{1/ p}\int_{Q_R[\Phi_0>t]}
\Phi_0\mathrm{dy}+\\
+t^{-p+1}\left(
\int_{Q_{R} [\overline\Psi>t]}\overline\Psi^p\mathrm{dy}+ 
\int_{\Sigma_{R}[\overline\varphi>t]}\overline\varphi^p\mathrm{ds_y}\right).
\end{align*}
Introduce the definitions of $\lambda$ and $\varkappa $ as in 
(\ref{lambd}) and (\ref{defva}), respectively.
For any $0<r< R_0$, and
$\omega\subset\subset Q=Q_{3/2}(0)$,
 keeping the same designation to the transformed
set $\omega\subset \subset  Q_{R_0}(z_0)$, we find
\begin{align*}
\left[{3{\rm dist}(\omega,\partial Q_{R_0}(z_0))\over 2R_0}\right]^{n(\gamma+1)/p}
\int_{\omega}\overline\Phi^{\gamma+1}\mathrm{dy}\leq 
{1\over \varkappa (p-1)-(\varkappa -1)\gamma}\times \\
\times\left[ (p-1)\left({3\over 2}\right)^n
\int_{Q}\overline\Phi^p\mathrm{dy}
+\varkappa \gamma 
\left(\int_{Q}\overline\Psi^{\gamma+1}\mathrm{dy}+ 
\int_{\Sigma}\overline\varphi^{\gamma+1}\mathrm{ds_y}\right)\right],
\end{align*}
for any $\gamma$
such that $p\leq \gamma+1<p+(p-1)/(\varkappa -1)$ and $\gamma+1\leq p+\delta$.
Therefore, setting $p+\varepsilon= \gamma+1$
we conclude (\ref{highup}).
\end{proof}

In a similar manner that we have Corollary \ref{high} 
from Proposition \ref{gia-geh}, Proposition \ref{surf} ensures the following 
Corollary.
\begin{corollary}\label{surft}
Under the conditions of Proposition \ref{surf},
if instead of (\ref{hisums}),
\begin{align*}
\frac 1{R^n}\int_{Q^+_{\alpha R}(z)} \Phi^p\mathrm{dx}
\leq B
\Big(\frac 1{R^n}\int_{Q^+_R(z)} \Phi^t \mathrm{dx}\Big)^{p/t}
+\frac 1{R^n}\int_{Q^+_R(z)} \Psi^p\mathrm{dx}+\nonumber
\\
+\frac 1{R^{n-1}}\int_{\Sigma_R(z)} \varphi^p\mathrm{ds},
\end{align*}
holds for $t<p$, then $\Phi\in  L^{p+\varepsilon}(Q_{r}^+(z_0))$,
for all $\varepsilon\in [0,\delta]\cap [0,(p-t)/(\varkappa -1)[$, with
 $\varkappa $ being
defined by (\ref{defva}), and $r=R_0/2$,
and it verifies
\begin{align}
\| \Phi\|_{p+\varepsilon,Q_r^+(z_0)}^{p+\varepsilon}
\leq R_0^{-n\varepsilon/p}
{(p-t)2^{n(p+\varepsilon)/p}\over p-t-(\varkappa -1)\varepsilon}
2^{3\varepsilon/p}\| \Phi\|_{p, Q_{R_0}^+(z_0)}^{p
+\varepsilon}+\nonumber\\
+{2^{n(p+\varepsilon)/p}\over p-t-(\varkappa 
-1)\varepsilon}
\left( 2^{(n+1)\varepsilon/p}(p-t)+
\varkappa (p-t+\varepsilon)\right)
\| \Psi\|_{p+\varepsilon, Q_{R_0}^+(z_0)}^{p+\varepsilon}+\nonumber\\
+
{2^{n(p+\varepsilon)/p}\over p-t-(\varkappa -1)\varepsilon}
\left(R_0(p-t)+\varkappa (p-t+\varepsilon)
\right)
 \| \varphi\|_{p+\varepsilon, \Sigma_{R_0}(z_0)}^{p+\varepsilon}
.\ 
\end{align}
\end{corollary}

\section{Local higher regularity of the gradient}
\label{schigh}

We separately study the interior and up to the boundary
the local higher regularity of the gradient of any
 weak solution to (\ref{omega})-(\ref{gama}).

\subsection{Interior higher regularity of the gradient}
\label{scint}

Let us begin by establishing a technical result.
\begin{lemma}\label{tecn}
For any $x\in\mathbb{R}^n$ and $R >0$, if 
 $\eta\in W^{1,\infty}_0(Q_R(x))$ is such that
 $0\leq \eta\leq 1$  and $|\nabla\eta|\leq 2/R$, then every 
$u\in W^{1,2n/(n+2)}(Q_R(x))$ verifies
\begin{equation}\label{qrx}
\| \eta(u-{-\hspace*{-0.4cm}}\int_{Q_R(x)}u\mathrm{dx})\|_{2,Q_R(x)}\leq  
 P_{\sqrt{n},2n/(n+2)}\|\nabla u\|_{{2n}/({n+2}),Q_R(x)},
\end{equation}
where
$P_{\sqrt{n},q}= S_{q}(1+2\sqrt{n}q\sin[\pi/q](q-1)^{-1/q}/\pi)$ if $q>1$,
$P_{\sqrt{2},1}=3S_1$ if $n=2$, and $S_{2n/(n+2)}=\pi^{-1/2}n^{(2-3n)/(2n)} (n-2)^{(n-2)/(2n)}[\Gamma(n)/
\Gamma(n/2)]^{1/n}$ (see Remark \ref{rsob}).
\end{lemma}
\begin{proof}
Making use of Remark \ref{rsob} with $q=2n/(n+2)<2\leq n$,
we obtain
\begin{align*}
\| \eta( u-{-\hspace*{-0.4cm}}\int_{Q_R(x)}u\mathrm{dx} )\|_{2,Q_R(x)}
\leq  S_{2n/(n+2)}\left(\|\nabla u\|_{2n/(n+2),Q_R(x)} + \right. \\ \left. +
\frac2R
\| u-{-\hspace*{-0.4cm}}\int_{Q_R(x)}u\mathrm{dx}\|_{2n/(n+2),Q_R(x)} \right).
\end{align*}
It is known that   the Poincar\'e constant, denoted by
 $C_{\Omega,q}$, that  stands for  the
optimal constant  in the Poincar\'e inequality, depends upon the value of $q$ 
and the geometry of the domain. 
Moreover, the Poincar\'e constant is at most 
 $dq\sin[\pi/q](q-1)^{-1/q}/(2\pi)$ \cite{ferone}
and $d/2$ if $q=1$ \cite{acosta}, if the domain is a bounded, convex, Lipschitz  with diameter $d $.
Hence, 
we conclude (\ref{qrx}).
\end{proof}

We show the interior local $W^{2+\varepsilon}$-estimate for any weak solution to the
problems under study (regardless of 
whether or not the solution under consideration
is subject to any boundary condition)
 if provided by data with higher integrability.
\begin{proposition}\label{int1}
If there exists $\delta>0$ such that
  ${\bf f}\in {\bf L}^{2+\delta}(\Omega)$, and  $f\in L^{2+\delta}(\Omega)$,
 then any function $u \in V_{2,\ell}$ solving
 (\ref{pbu}) or (\ref{pbun}) belongs to $ W^{1,2+\varepsilon}_{\rm loc}(\Omega)$,
for all $\varepsilon\in [0,\delta]\cap [0,4 ((n+2)(\upsilon_{\rm I}-1)) ^{-1}[$, where 
\[\upsilon_{\rm I}=
 (8^n+1) 2^{6n}\left[2P_{\sqrt{n},2n/(n+2)}
\left( {2\over a_\#}\right)^{1/2}
\left( {4(a^\#)^2\over a_\#}+1+{\nu_3\over 2} \right)^{1/2}
+1\right]^2,\]
with 
 $\nu_3=\nu_3(f)$ being positive constant if 
  $f\not= 0$, and  $\nu_3(0)=0$ otherwise. In particular, 
for every $x_0\in\Omega$ and $0<r<
{\rm dist}(x_0,\partial\Omega)/\sqrt{n}$,
 we have
\begin{align}
\| \nabla u\|^{2+\varepsilon}_{2+\varepsilon,Q_{r/2}(x_0)}
\leq  r^{-n\varepsilon/2} 2^{\varepsilon/2}
Z_1(\upsilon_{\rm I})
\| \nabla u\|_{2, Q_{r}(x_0)}^{2+\varepsilon}
+ \nonumber\\ +
\left(2^{(n+1)\varepsilon/2}
Z_1(\upsilon_{\rm I})+
Z_2(\upsilon_{\rm I})\right) 
\| {\mathcal F}(a_\#)\|_{2+\varepsilon, Q_{r}(x_0)}^{2+\varepsilon},
\label{ri}
\end{align}
with ${\mathcal F}(a_\#)$, $Z_1$ and $Z_2$ being the functions defined in 
(\ref{fa}), (\ref{zz1}) and  (\ref{zz2}),
 respectively.
\end{proposition}
\begin{proof}
Let  $\eta\in W^{1,\infty}(\mathbb{R}^n)$, and $U\in \mathbb{R}$.
 Taking $v=\eta^2(u-U) \in V_{2,\ell}$ as a test function in (\ref{pbu})
 or (\ref{pbun}),
 applying the H\"older inequality, and using (\ref{amin})-(\ref{amax}),  
we obtain
 \begin{align}
 {a_\#}\|\eta \nabla u\|_{2,\Omega}^2
\leq 2a^\# \|\eta\nabla u\|_{2,\Omega}
\| (u-U)\nabla \eta \|_{2,\Omega}+\nonumber\\+
 \|\eta{\bf f}\|_{2,\Omega}\left(\|\eta\nabla u\|_{2,\Omega}+
2\| (u-U)\nabla \eta \|_{2,\Omega}\right)
 +\|\eta f\|_{2,\Omega}\|\eta (u-U)\|_{2,\Omega}.\label{pbuab}
\end{align}

Fix $z\in\Omega$ and $0<r<R\leq R_0$
 such that  $Q_{R_0}(z)\subset\subset\Omega$.
Choosing  $\eta\equiv 1$ in $Q_r(z)$, $\eta\equiv 0$ in $\mathbb{R}^n
\setminus Q_R(z)$,  $0\leq \eta\leq 1$ in $\mathbb{R}^n$,
 and $|\nabla\eta|\leq 1/(R-r)$ a.e. in $Q_R(z)$. 
 For any  $\alpha\in [1/2,1[$, choosing $r=\alpha R$, we have $R-r<1$.
Thus, applying the Young inequality, we get
 \begin{align*}
 {a_\#}(1-\nu_0-\nu_1)\|\nabla u\|_{2,Q_{\alpha R}(z)}^2\leq
{1\over (1-\alpha)^2R^2} \left( {(a^\#)^2\over\nu_0 a_\#}+
\nu_2+{\nu_3\over 2} \right)\|\eta (u-U)\|_{2,Q_R(z)}^2
+\\+
\left( {1\over 4\nu_1a_\#}+ {1\over\nu_ 2}\right)\|{\bf f}\|_{2,Q_R(z)}^2
+{1\over 2\nu_3}\|f\|_{2,Q_R(z)}^2.
\end{align*}
Next, taking $\alpha=1/2$, applying Lemma \ref{tecn} with
 $U=(2R)^{-n}\int_{Q_R(x)}u\mathrm{dx}$,
and divided by $R^n$, we obtain
\begin{align*}
\frac1{R^n}\int_{Q_{ R/2}(z)}|\nabla{  u}|^2 \mathrm{dx}\leq
B \Big(\frac1{R^n} \int_{Q_R(z)}|\nabla u|^{2n/(n+2)}\mathrm{dx}
\Big)^{n+2\over n}+\\ +{1\over a_\#(1-\nu_ 0-\nu_ 1)}
\frac1{R^n}\Big(\int_{Q_R(z)}
\left( {1\over 4\nu_1 a_\#}+ {1\over\nu_ 2}\right)|{ \bf f}|
^{2}\mathrm{dx}+\int_{Q_{R}(z)}{1\over 2\nu_3}
|{  f}|^{2}\mathrm{dx}\Big),
\end{align*}
with $\nu_0+\nu_1<1$,  and 
\begin{equation}\label{defbb}
B= {4\over a_\#(1-\nu_ 0-\nu_ 1)}
\left( {(a^\#)^2\over \nu_0a_\#}+\nu_2+{\nu_3\over 2} \right)
(P_{ \sqrt{n} ,2n/(n+2)})^2
.\end{equation}

Applying  Corollary \ref{high} with
$\Phi=|\nabla u|$, 
 $t=2n/(n+2)$, $p=2$,
\begin{equation}\label{psi2}
\Psi=\left({({1/(4\nu_1 a_\#)}+ 1/\nu_2)
|{\bf f}|^{2}+|f|^{2}/(2\nu_3)\over a_\#(1-\nu_ 0-\nu_ 1)}
\right)^{1/2} \in L^{2+\delta}(\Omega),
\end{equation}
and taking $\nu_0 =\nu_1=1/4$ and $\nu_2=1$, we conclude the claim.
\end{proof}

\subsection{Higher regularity up to the boundary of the gradient}
\label{scup}

Let us recall the general definition of  $C^{k,\lambda}$ domain.
\begin{definition}\label{cka}
We say that 
$\Omega$ is a domain of class $C^{k,\lambda}$ (or simply  $C^{k,\lambda}$ domain),
$k\in \mathbb N_0$ and $ \lambda\in [0,1],$ if
$\Omega$ is an open, bounded, connected, nonempty set 
 of $\mathbb R^n$ and it verifies the following:
\begin{equation}
\exists M\in\mathbb N\quad
\exists  \varrho, \nu>0:\qquad\partial\Omega=\cup^M_{m=1}\Gamma_m,
\end{equation}
with 
\begin{enumerate} 
\item 
$\Gamma_m=O^{-1}_m(\{y=(y',y_{n})\in
Q^{(n-1)}_\varrho(0)\times\mathbb{R}:$ $y_{n}=\varpi_m(y')\}$,
\item
$O^{-1}_m(\{y=(y',y_{n})\in Q^{(n-1)}_\varrho(0)\times\mathbb{R}: $
$\varpi_m(y')<y_{n}<\varpi_m(y')+\nu\}) \subset\Omega,$
\item
$O^{-1}_m(\{y=(y',y_{n})\in Q^{(n-1)}_\varrho(0)\times\mathbb{R}:$
$\varpi_m(y')-\nu<y_{n}<\varpi_m(y')\}) \subset \mathbb R^n\setminus \Omega,$
\end{enumerate} 
where
\[
Q^{(n-1)}_\varrho(0)=\{y'=(y_{1},\cdots,y_{(n-1)})\in
\mathbb{R}^{n-1}:\ |y_{i}|<\varrho,\ i=1,\cdots,n-1\},
\]
and for each  $m=1,\cdots,M$, $ O_m:\mathbb R^n\rightarrow \mathbb R^n$ 
denotes a local coordinate system:
$$y^{(m)}=O_m(x)= \mathsf{O} x+b,\quad 
 \mathsf{O}^{-1}= \mathsf{O}^T,\ \det \mathsf{O}=1;$$
and $ \varpi_m\in C^{k,\lambda}(Q^{(n-1)}_\varrho(0))$.
\end{definition}

\begin{proposition}\label{propup}
Let $\Omega$ be a $C^{1}$ domain. 
 If there exists $\delta>0$ such that
  ${\bf f}\in {\bf L}^{2+\delta}(\Omega)$, $f\in L^{2+\delta}(\Omega)$, 
$g\in L^{2+\delta}(\Gamma_{\rm N})$ and
$h\in L^{2+\delta}(\Gamma)$, then 
for every $x_0\in\partial\Omega$ there exists a cube $Q_0\subset\mathbb{R}^n$
of side length $2R_0$ centered at the point $x_0$ such that 
any function $u \in V_{2,\ell}$ solving (\ref{pbu}) verifies
\begin{align}
\| \nabla u\|^{2+\varepsilon}_{2+\varepsilon,Q_{r/2}\cap\Omega}
\leq r^{-n\varepsilon/2}2^{3\varepsilon/2}Z_1(\upsilon_{\rm U})
\| \nabla u\|_{2, Q_r\cap\Omega}^{2+\varepsilon}+\nonumber\\
+\left( 2^{(n+1)\varepsilon/2}Z_1(\upsilon_{\rm U})
+Z_2(\upsilon_{\rm U}) \right)
\| {\mathcal F }(a_\#)\|_{2+\varepsilon, Q_r\cap\Omega}^{2+\varepsilon}
+\nonumber\\
+
\left( Z_1(\upsilon_{\rm U})+ Z_2(\upsilon_{\rm U})\right)
 \| {\mathcal H}_0(a_\#,b^\#)
\|_{2+\varepsilon, Q_r\cap\partial\Omega}^{2+\varepsilon},
\label{ru}
\end{align}
 for any cube $Q_r\subset\subset Q_0$ of radius $r$ centered at $x_0$,
 \begin{equation}\label{defuu}
\upsilon_{\rm U}=
 (8^n+1) 2^{6n}\left[2P_{\sqrt{n},2n/(n+2)} \left({2\over a_\#}\right)^{1/2}
\left( {8(a^\#)^2\over a_\#}+2+{\nu_3\over 2} \right)^{1/2}
+1\right]^2,
\end{equation}
 where  $\nu_3=\nu_3(f)$ is a positive constant if   $f\not=0$,
and $\nu_3(0)=0$ otherwise, and 
\[
 {\mathcal H}_0(a_\#, b^\#) ={2K_{{2n/( n+1)}}
\over (a_\#)^{1/2}}
\left({2\over  a_\#}+{ 2^{-1/n}}\right)^{1/2} 
|g\chi_{\Gamma_{\rm N}}+\left(h+b^\# [{\rm ess}\sup_{Q_r\cap\Omega}
 |u|] ^{\ell -1}\right) \chi_{\Gamma}|.
\]
Here ${\mathcal F}(a_\#)$,
$Z_1$ and $Z_2$ are the functions defined in (\ref{fa}), (\ref{zz1}) and  (\ref{zz2}),
 respectively.
\end{proposition}
\begin{proof}
Let   $U\in \mathbb{R}$, and $\eta\in W^{1,\infty}_0(\mathbb{R}^n)$ satisfy
 $0\leq \eta\leq 1$ in $\mathbb{R}^n$.
 Taking $v=\eta^2 (u-U) \in V_{2,\ell}$ as a test function in (\ref{pbu}),
 applying the H\"older inequality, and making use of (\ref{amin})-(\ref{amax}),
the monotone property of $b$, and (\ref{bmax}),  
we obtain
 \begin{align}
 {a_\#}\|\eta \nabla u\|_{2,\Omega}^2
\leq 2a^\# \|\eta^2 \nabla u\|_{2,\Omega}
\|(u-U)\nabla \eta \|_{2,\Omega}  +
 \nonumber\\+
 \|\eta{\bf f}\|_{2,\Omega}\left(\|\eta\nabla u\|_{2,\Omega}+
2\| (u-U)\nabla \eta \|_{2,\Omega}\right)
 +\|\eta f\|_{2,\Omega}\| \eta(u-U) \|_{2,\Omega}+\nonumber\\
 +\|\eta (g\chi_{\Gamma_{\rm N}}
 +\left( h+ b^\# |U|^{\ell-1}\right) \chi_{\Gamma})\|_{2,\partial\Omega}\|\eta (u-U) \|_{2,\partial\Omega}.\label{fgh}
\end{align}

By Definition \ref{cka}, 
there exist $M \in\mathbb{N}$ and $\varrho,\nu>0$ such that
for any $x_0\in\partial\Omega$ there is  $ m\in\{1,\cdots,M\}$ such that a
local coordinate system $y^{(m)}=O_m(x)$
and a local $C^{1,1}$-mapping $\varpi_m$ verify 
\begin{equation}\label{gm}
x_0\in \Gamma_m=O_m^{-1}\circ \phi_m^{-1}\left(
Q_\varrho^{(n-1)}(0)\times\{0\}\right),
\end{equation}
where 
$\phi_{m}: Q_\varrho^{(n-1)}(0)\times \mathbb{R}
\rightarrow \mathbb R^n$ of class $C^{1,1}$ is defined by
\begin{equation}
\phi_{m}(y)=
\left(
\begin{array}c
y' \\
y_n-\varpi_{m}(y')
\end{array}
\right).
\label{phi}
\end{equation}
We use the notation $y'=(y_1,\ldots, y_{n-1})\in \mathbb R^{n-1}$.
For each  $ m\in\{1,\cdots,M\}$,
we consider  the change of variables
\begin{equation}\label{chg}
z\in Q_\varrho^{(n-1)}(0)\times ]-\nu,\nu [
\mapsto y=\phi_m^{-1}(z)\mapsto x=O^{-1}(y).
\end{equation}
Since the Jacobian of  the transformation $O^{-1}_m\circ
\phi^{-1}_m$ is equal to 1, let us denote by the same letter any function
$f=f\circ O^{-1}_m\circ\phi^{-1}_m$.

Fix $ m\in\{1,\cdots,M\}$ such that $x_0\in\Gamma_m$ is
in accordance with (\ref{gm}),
set $z_0=\phi_m\circ O_m(x_0)$, and 
\[
\Sigma_R(z_0)=\{z\in  Q_\varrho^{(n-1)}(0)\times ]-\nu,\nu [:\
 |z'-z_0'|<R, \ z_n=0\}, 
\]
for any $0<R\leq \min\{\varrho,\nu\}$.
 Notice that $z_0=(z_0',0)$.

Let  $0<r<R\leq R_0=\min\{\varrho,\nu,{\rm dist}(z_0',\partial'
 Q_\varrho^{(n-1)}(0))\}$.
 We choose
 $\eta\equiv 1$ in $Q_r(z_0)$, $\eta\equiv 0$ in $\mathbb{R}^n
\setminus Q_R(z_0)$, and $|\nabla\eta|\leq 1/(R-r)$
a.e. in $Q_R(z_0)\setminus Q_r(z_0)$.

By Remark \ref{rk} and
$u\eta\in W^{1,2n/(n+1)}(Q_R^+(z_0))$ with $2n/(n+1)<2\leq n$, 
making use of the H\"older inequality, we get
\begin{align*}
\|\eta (u-U)\|_{2,\Sigma_R(z_0)}
\leq K_{{2n/( n+1)}}
|Q_R^+(z_0)|^{{1\over 2n}}\left(
\|\eta\nabla u\|_{2,Q_R^+(z_0)}+ \right.\\
\left.+
\|(u-U)\nabla\eta\|_{2,Q_R^+(z_0) } + \|\eta(u-U)\|_{2,Q_R^+(z_0)}\right),
\end{align*}

Inasmuch as $\Gamma_0=O_m^{-1}\circ \phi_m^{-1}\left(
\Sigma_{R_0}(z_0)\right)$, different cases occur, namely
 $\Gamma_0\cap\Gamma\not=\emptyset$
and  $\Gamma_0\cap\Gamma_{\rm N}\not=\emptyset$;
 $\Gamma_0\subset\Gamma$, and 
 $\Gamma_0\subset\Gamma_{\rm N}$. Throughout the sequel, we refer to
 $\|\cdot\|_{2,\Sigma_R(z_0)}$ including cases where the set is empty.

Thus, the transformed last term in (\ref{fgh})  is analyzed as follows
\begin{align*}
 \|\eta  (g\chi_{\Gamma_{\rm N}}
 + \left( h+ b^\# |U|^{\ell-1}\right) \chi_{\Gamma})\|_{2,\Sigma_R(z_0)}
\|\eta (u-U) \|_{2,\Sigma_R(z_0)}\leq \\ \leq
\nu_4a_\#\|\eta\nabla u\|_{2,Q_R^+(z_0)}^2+
{\nu_5\over (R-r)^2}\|\eta (u-U)|_{2,Q_R^+(z_0)}^2 +\\
+R(K_{{2n/( n+1)}})^2\left( {1\over 2\nu_4a_\#}+{ 2^{(n-1)/n}\over\nu_5}
 \right)
\|g\chi_{\Gamma_{\rm N}}+\left( h+ b^\# |U|^{\ell-1}\right) 
\chi_{\Gamma}\|_{2,\Sigma_R(z_0)}^2.
 \end{align*}

From above
 and reorganizing the other terms in (\ref{fgh})
as in  the proof of Proposition \ref{int1}, we have
 \begin{align*}
\|\nabla u\|_{2,Q_{r}^+(z_0)}^2
\leq
{B\over (R-r)^2} \|\eta (u-U)\|_{2,Q_R^+(z_0)}^2
+\\ +
{1\over  a_\#(1-\nu_0-\nu_1-\nu_4)}\left[
\left( {1\over 4\nu_1a_\#}+ {1\over \nu_2}\right)\|{\bf f}\|_{2,Q_R^+(z_0)}^2
+{1\over 2\nu_3}\|f\|_{2,Q_R^+(z_0)}^2+\right.\\
\left.+
R(K_{{2n/( n+1)}})^2\left( {1\over 2\nu_4a_\#}+{ 2^{(n-1)/n}\over\nu_5}
 \right)
 \|g\chi_{\Gamma_{\rm N}}+ \left( h+ b^\# |U|^{\ell-1}\right) 
\chi_{\Gamma}\|_{2,\Sigma_R(z_0)}^{2}\right],
\end{align*}
where instead of (\ref{defbb}) we use
\[
B= {4\over a_\#(1-\nu_ 0-\nu_ 1-\nu_4)}
\left( {(a^\#)^2\over \nu_0a_\#}+\nu_2+{\nu_3\over 2}+\nu_5
 \right)(P_{\sqrt{n},2n/(n+2)})^2.
 \]
Employing Lemma \ref{tecn} into the above inequality,
we apply Corollary \ref{surft}
with $\Phi=|\nabla u|$,  $t=2n/(n+2)$, $p=2$, and
\begin{align*}
\Psi=& \left({({1/(4\nu_1 a_\#)}+ 1/\nu_2)
|{\bf f}|^{2}+|f|^{2}/(2
\nu_3)\over a_\#(1-\nu_ 0-\nu_ 1-\nu_4)}
\right)^{1/2} \in L^{2+\delta}(Q_R^+(z_0));
\\
\varphi=&\left(
{{1/(2\nu_4 a_\#)}+ 2^{1-1/n}/\nu_5
\over a_\#(1-\nu_ 0-\nu_ 1-\nu_4)}\right)^{1/2}K_{{2n}/({n+1})} \times \\
& \times
|g\chi_{\Gamma_{\rm N}}+\left( h+ b^\# |U|^{\ell-1}\right) \chi_{\Gamma}|
\in L^{2+\delta}(\Sigma_R(z_0)).
\end{align*}
Upon choosing $\nu_1=1/4$, $\nu_0=\nu_4=1/8$, $\nu_2 =\nu_5=1$, 
and $Q_0=O_m^{-1}\circ \phi_m^{-1}\left(
Q_{R_0}(z_0)\right)$,
 the application of the passage to the initial coordinates system finishes
the proof of Proposition \ref{propup}.
\end{proof}

\begin{proposition}\label{neumann}
Under the conditions of Proposition \ref{propup} 
 such that the compatibility condition (\ref{cc}) is verified,
 provided that  $\Gamma_{\rm N}=\partial\Omega$,
any function $u \in V_{2}$ solving (\ref{pbun}) verifies
 (\ref{ru}) with $ {\mathcal H}$ being replaced by
$ {\mathcal G}$ which is defined by  (\ref{gan}).
\end{proposition}
\begin{proof}
Let  $\eta\in W^{1,\infty}_0(\mathbb{R}^n)$ satisfy
 $0\leq \eta\leq 1$ in $\mathbb{R}^n$, and $d=-\int_\Omega u\eta^2 \mathrm{dx}$.
 Taking $v=u\eta^2+d\in V_{2}$ as a test function in (\ref{pbun}),
 applying the H\"older inequality, and making use of
(\ref{cc}) and (\ref{amin})-(\ref{amax}),  
we obtain
 \begin{align*}
 {a_\#}\|\eta \nabla u\|_{2,\Omega}^2\leq 2a^\# \|\eta\nabla u\|_{2,\Omega}
\| u\nabla \eta \|_{2,\Omega}
 +\|\eta f\|_{2,\Omega}\|u\eta \|_{2,\Omega} +\nonumber\\+
 \|\eta{\bf f}\|_{2,\Omega}\left(\|\eta\nabla u\|_{2,\Omega}+
2\| u\nabla \eta \|_{2,\Omega}\right)
 +\|\eta g\|_{2,\Gamma_{\rm N}}\|u\eta \|_{2,\Gamma_{\rm N}}.
\end{align*}

The argument of the proof of Proposition \ref{propup} may be reproduced,
 finishing the proof of Proposition \ref{neumann}.
\end{proof}

\section{Existence results}
\label{sce}

\subsection{Preliminary results}
\label{pre}

Making recourse of the
Poincar\'e inequality \cite[Corollary 3]{chakib}:
\[
\|v\|_{q,\Omega}\leq P_q\left(\sum_{i=1}^n
\|\partial_iv\|_{q,\Omega}+|\Gamma|^{1/q-1}
\left|\int_\Gamma v\mathrm{ds}\right|\right),
\]
 we introduce $S_{q,\ell}=S_q\max\{1+ P_q2^{(n-1)(1-1/q)},
P_q |\Gamma |^{1/q-1/\ell}\}$ and $K_{q,\ell}=K_q \max\{1+ P_q
2^{(n-1)(1-1/q)},
P_q |\Gamma |^{1/q-1/\ell}\}$ that verify
\begin{align*}
\|v\|_{nq/(n-q),\Omega}\leq S_{q,\ell}\|v\|_{1,q,\ell}
;\\  
\|v\|_{(n-1)q/(n-q),\partial\Omega}\leq K_{q,\ell}\|v\|_{1,q,\ell}.
\end{align*}

Let us recall two standard existence results that we apply later \cite{lc-ijpde,sinica}.
\begin{proposition}[meas$(\Gamma)>0$]\label{exist}
Let   ${\bf f}\in {\bf L}^{2}(\Omega)$,  $f\in L^{2}(\Omega)$, 
$g\in L^{s}(\Gamma_{\rm N})$, and $h\in L^{\ell\,'}(\Gamma)$, with $s,\ell\geq 2$.
Under the assumptions (A)-(B), there exists $u \in V_{2,\ell}$ 
being a weak solution to (\ref{omega})-(\ref{gama}), {\em i.e.}
 solving (\ref{pbu}) for all $v\in  V_{2,\ell}$. Moreover,
 the following estimate holds
\begin{align}
 {a_\#\over 2}\|\nabla u\|_{2,\Omega}^2+
{ b_\#\over\ell\,'}\| u\|_{\ell,\Gamma}^\ell\leq
{\ell-1\over\ell b_\#^{1/(\ell-1)}}\left(\|h\|_{\ell\,',\Gamma}+
\mathcal{E}(1,1)\right)^{\ell/(\ell-1)}+
\nonumber\\
+  {1\over 2a_\#}\left(
\|{\bf f}\|_{2,\Omega}+\mathcal{E}( |\Omega |^{1/n},|\Omega |^{1/2+(1/n-1)/s})
\right)^2,
\label{cotau2}
\end{align}
where $\mathcal{E}(A,B)=AS_{2n/(n+2),\ell}\|f\|_{2,\Omega}+
BK_{ns/ (n(s-1)+1 ),\ell}
\|g\|_{s,\Gamma_{\rm N}}$. 
\end{proposition}
\begin{proof}
Since the existence and uniqueness are classical, we only pay attention  on
the derivation of (\ref{cotau2}).
 Taking $v=u\in V_{2,\ell}$ as a test function in (\ref{pbu}),
 and making use of (\ref{amin}) and (\ref{bmin}),  
we obtain
 \begin{align*}
 {a_\#}\| \nabla u\|_{2,\Omega}^2+b_\#\|u\|_{\ell,\Gamma}^\ell
\leq 
 \|{\bf f}\|_{2,\Omega}\|\nabla u\|_{2,\Omega}+
\|h\|_{\ell\,',\Gamma}\|u\|_{\ell,\Gamma}+\\
 +\| f\|_{2,\Omega}S_{2n/(n+2),\ell}\left(|\Omega |^{1/n}
\|\nabla u \|_{2,\Omega}+\|u\|_{\ell,\Gamma}
\right)+\nonumber\\
 +\| g 
\|_{s,\Gamma_{\rm N}} K_{ns/( n(s-1)+1),\ell}
\left(|\Omega |^{1/2+(1/n-1)/s}
\|\nabla u \|_{2,\Omega}+\|u\|_{\ell,\Gamma}
\right),
\end{align*}
 applying the H\"older inequality, the Sobolev embedding $W^{1,2n/(n+2)}(\Omega)
\hookrightarrow L^2(\Omega)$, the trace embedding $W^{1,ns/ (n(s-1)+1) }(\Omega)
\hookrightarrow L^{s'}(\partial\Omega)$. This
completes the proof of Proposition \ref{exist}.
\end{proof}

\begin{remark}
From the definition of the Gamma function,
$\Gamma(n/2+1)=(n/2)!$ if $n$ is even, and 
$\Gamma(n/2+1)=\pi^{1/2}2^{-(n+1)/2}n(n-2)(n-4)\cdots 1$ if $n$ is odd,
the two-dimensional constant $S_1$ is simply $\pi^{-1/4}2^{-3/2}$.
\end{remark}

\begin{proposition}[meas$(\Gamma)=0$]\label{existn}
Let   ${\bf f}\in{\bf L}^{2}(\Omega)$,  $f\in L^{2}(\Omega)$, 
 $g\in L^{s}(\Gamma_{\rm N})$, with $s\geq 2$, such that the compatibility condition
(\ref{cc}) is verified.
Under the assumption (A), there exists $u \in H^1(\Omega)$ 
being the unique function such that $\int_\Omega u\mathrm{dx}=0$,
 solving (\ref{pbun}) for all $v\in  V_{2}$. Moreover,
 the following estimate holds
\begin{align}
 \|\nabla u\|_{2,\Omega}\leq {1\over 
a_\#}\left(
\|{\bf f}\|_{2,\Omega}+|\Omega |^{1/n}  S_{2n/(n+2)}
\|f\|_{2,\Omega}+\right.\nonumber\\
\left. + |\Omega |^{1/2+(1/n-1)/s} K_{ns/ (n(s-1)+1) }
\|g\|_{s,\Gamma_{\rm N}}\right). 
\label{cotaun2}
\end{align}
\end{proposition}

\subsection{Proof of Theorem \ref{main1}}
\label{scm1}

Supposing that the conditions of Proposition \ref{exist} are fulfilled,
 there exists $u \in V_{2,\ell}$ 
being a weak solution to (\ref{omega})-(\ref{gama}), {\em i.e.}
 solving (\ref{pbu}) for all $v\in  V_{2,\ell}$.

On the one hand, Proposition \ref{int1} ensures that
for  each point $x\in\Omega$
it is associated a sequence of cubic intervals $Q_{r(x)/2}(x)$,
with side lengths $r(x)>0$  tending to zero, such that
(\ref{ri}) is verified.
On the other hand, Proposition \ref{propup} ensures that
for  each point $x\in\partial\Omega$
it is associated a sequence of cubic intervals $Q_{r(x)/2}(x)$,
with side lengths  $r(x)>0$  tending to zero, such that
(\ref{ru}) is verified.

 Let us denote the collection of the above balls by $\mathcal B$, {\em i.e.}
 \[\mathcal{B}=\{ Q_{r_k/2}(x)\}_{x\in\overline\Omega,\ k\geq 1}.
\]
 All
 radii of balls in  $\mathcal B$ are totally bounded by $(8S_{2n/(n+2)})^{-1}$.
According to the Besicovitch covering theorem
 \cite[Theorem 1.2]{guz}, there exists a sequence of cubic intervals
$\{Q_{r_m/2}(x^{(m)})\}_{m\geq 1}$ in  $\mathcal B$ such that: $\overline{\Omega}\subset \cup_{m\geq 1}Q_{r_m/2}(x^{(m)})$; and
every point of $\mathbb{R}^n$ belongs to at most $2^n+1$ balls in 
$\{Q_{r_m/2}(x^{(m)})\}_{m\geq 1}$.

Since $\Omega$
 is bounded, its closure $\overline\Omega$
is compact. Hence it can be covered with finitely many cubic intervals
$Q_{r_m/2}(x^{(m)})$, $m=1,\cdots, M$. Let us define 
\begin{equation}\label{rstm}
r_\#=\min\{r_m:\ m=1,\cdots,M\}.\end{equation}

Setting $\mathcal{C}=\{Q_{r_m}(x^{(m)})\}_{m=1,\cdots, M}$, we build
the collection $\mathcal{C}_1$ as being the union of $Q^{(1)}_1=Q_{r_1}(x^{(1)})$
with all pairwise disjoint cubes $Q^{(1)}_m=Q_{r_m}(x^{(m)})\in \mathcal{C}$ 
 such that $Q^{(1)}_1\cap Q^{(1)}_m=\emptyset$.
 This selection process may be recursively repeated, by building
 $\mathcal{C}_k$ as being the union of $Q^{(k)}_1\in \mathcal{C}\setminus
 \mathcal{C}_{k-1}$
with all pairwise disjoint cubes $Q^{(k)}_m\in \mathcal{C}\setminus
 \mathcal{C}_{k-1}$ 
 such that $Q^{(k)}_1\cap Q^{(k)}_m=\emptyset$.
 Consequently,  there exists a number $N$, depending on the dimension of the space,
such that for each $k\in\{1,\cdots, N\}$, we collect pairwise disjoint cubes
corresponding to half ones from $\mathcal{B}$. 

For each $k\in\{1,\cdots, N\}$, we split the set of indices as
$\mathcal{I}(k)\cup \mathcal{J}(k)$,
where $\mathcal{I}(k)$ contains the indices with $x^{(i)}\in\Omega$,
while $\mathcal{J}(k)$ contains the indices with $x^{(j)}\in\partial\Omega$.
Hence, combining (\ref{ri}) and (\ref{ru}) with
\[
\|\nabla u\|_{p,\Omega}^p\leq \sum_{k=1}^N\left(\sum_{i\in\mathcal{I}(k)}
\|\nabla u\|_{p,Q_{r_i/2}(x^{(i)})}^p+\sum_{j\in\mathcal{J}(k)}
\|\nabla u\|_{p,Q_{r_j/2}(x^{(j)})\cap\Omega} ^p\right),
\]
 we find (\ref{cotam1}).

\subsection{Proof of Theorem \ref{main1n}}

Applying Propositions \ref{neumann} and \ref{existn}
instead of Propositions \ref{propup} and \ref{exist}, respectively,
this proof is {\em mutatis mutandis} the proof of Theorem \ref{main1}.

\section{Proof of Theorem \ref{main2}}
\label{secm2}

We consider the operator $\mathcal{T}:V_{p,\ell}\rightarrow V_{p,\ell}$,
for any $p\in [2,2+1/(\upsilon-1)[$ and $\ell\geq 2$,
defined by
\[
\theta\mapsto \phi=\phi(\theta)\mapsto \Theta,
\]
where $\phi\in W^{1,p}(\Omega)$ is the unique solution,
verifying $\int_\Omega \phi\mathrm{dx}=0$, to the auxiliary
electric problem
\begin{equation}\label{auxe}
\int_\Omega \sigma(\theta)\nabla\phi\cdot\nabla w\mathrm{dx}=-
\int_\Omega\alpha_{\rm s}(\theta) \sigma(\theta)\nabla\theta\cdot\nabla w\mathrm{dx}+
\int_{\Gamma_{\rm N}} gw\mathrm{ds},\quad\forall w\in V_{p'},
\end{equation}
and $\Theta\in V_{p,\ell}$  is the unique solution to the auxiliary
thermal problem
\begin{align}\label{auxt}
\int_\Omega k_{}(\theta)\nabla\Theta\cdot\nabla v\mathrm{dx}+
\int_{\Gamma}f_\lambda(\theta)|\Theta| ^{\ell-2}\Theta v\mathrm{ds}=
\int_{\Gamma}\gamma(\theta)\theta_{\rm e}^{\ell-1} v\mathrm{ds}\nonumber\\
-
\int_\Omega \sigma(\theta)\left(\alpha_{\rm s}(\theta)(\theta+\phi)
\nabla\theta+\phi\nabla\phi\right)\cdot\nabla v\mathrm{dx},
\quad\forall v\in V_{p',\ell}.
\end{align}

The existence and  uniqueness of $\phi$ and $\Theta$ are ensured by
\begin{itemize}
\item Theorem \ref{main1n}, under $\mathsf{A}=\sigma(\theta)$,
${\bf f}=\alpha_{\rm s}(\theta) \sigma(\theta)\nabla\theta
\in {\bf L}^{p}(\Omega)$, and $g\in L^p(\Gamma_{\rm N}).$
\item Theorem \ref{main1}, under $\mathsf{A}=k(\theta)$,
${\bf f}=\alpha_{\rm s}(\theta)(\theta+\phi) \sigma(\theta)
\nabla\theta+\phi\sigma(\theta) \nabla\phi
\in {\bf L}^{p}(\Omega)$, and $h=
\gamma(\theta)\theta_{\rm e}^{\ell-1}\in L^p(\Gamma).$
The uniqueness of $\Theta$ is true under the strict monotone property 
according to Remark \ref{rmo}.
\end{itemize}

Next, let us prove that $\mathcal{T}$ maps the closed ball
$\overline B_R(0)$ into itself.
Let $\theta\in\overline B_R(0)$, {\em i.e.} $\theta\in V_{p,\ell}$ satisfies
$\|\nabla\theta\|_{p,\Omega}+\|\theta\|_{\ell,\Gamma}\leq R$.
Consequently, we have
\[
\|\theta\|_{\infty,\Omega}\leq C_\infty R.
\]

 Using the  estimates (\ref{cotaun2}) and (\ref{cotam1})
in accordance with Theorem \ref{main1n},
and taking $R\geq \|g\|_{p,\Gamma_{\rm N}}$, we deduce
 \begin{align}
 \|\nabla\phi\|_{2,\Omega}\leq {1\over \sigma_\#}\left(\sigma^\#\alpha^\#
 \|\nabla\theta\|_{2,\Omega}+|\Omega|^{1/(2p')} K_{2p/(2p-1)}
\|g\|_{p,\Gamma_{\rm N}}\right);\nonumber \\
 \|\nabla\phi\|_{p,\Omega}\leq \mathcal{M}_1 \|g\|_{p,\Gamma_{\rm N}}
+\mathcal{M}_2 \|\nabla\theta\|_{p,\Omega} \leq
(\mathcal{M}_1+\mathcal{M}_2)R,\label{cotapp}
 \end{align}
 with
 \begin{align*}
  \mathcal{M}_1&=
2^{3/2}  5^{1/p} \frac {|\Omega|^{1/(2p')}K_{2p/(2p-1)}
(r_\#)^{2/p-1}+[1+\upsilon (p-1)]^{1/p}\sqrt{1+\sigma_\#}}
{ (p-1-\upsilon(p-2))^{1/p}\sigma_\#} ; \\
\mathcal{M}_2&=
2^{3/2}   5^{1/p} \sigma^\#\alpha^\#  
\frac{|\Omega|^{{1\over 2}-{1\over p}}(r_\#)^{{2\over p}-1}
+[2^{(n+1)(p-2)/2}+\upsilon (p-1)]^{1/p}\sqrt{1+\sigma_\#}}
 {(p-1-\upsilon(p-2))^{1/p} \sigma_\#}, 
  \end{align*}
  considering that $K_{4/3}=1/\pi$, and $2<p<3$.
  
 Using the  estimates (\ref{cotau2}) and (\ref{cotam1}), we deduce
 \begin{align*}
{k_\#\over 2} \|\nabla\Theta\|_{2,\Omega}^2+
{b_\#\over \ell\,'}\|\Theta\|_{\ell,\Gamma}^\ell
\leq {(\gamma^\#)^{\ell\,'}\over \ell\,' (b_\#)^{1/(\ell-1)}}
\|\theta_\mathrm{e}\|_{\ell,\Gamma}^{\ell}+\\
+{(\sigma^\#)^2 |\Omega|^{1-2/p}\over 2k_\#}\left(\alpha^\#(
\|\theta\|_{\infty,\Omega} +\|\phi\|_{\infty,\Omega} )\|\nabla\theta\|_{p,\Omega}
+\|\phi\|_{\infty,\Omega}  \|\nabla\phi\|_{p,\Omega}\right)^2;\\
 \|\nabla\Theta\|_{p,\Omega}\leq \mathcal{M}_3
\|\theta_\mathrm{e}\|_{\ell,\Gamma}^{\ell/2}
 +\mathcal{M}_4 \left( \gamma^\#
\|\theta_\mathrm{e}\|_{(\ell-1)p,\Gamma}^{\ell-1}
+b^\#[{\rm ess} \sup_\Omega |\Theta | ]^{\ell-1} |\Gamma|^{1/p}
\right) +\nonumber\\
 +\mathcal{M}_5 \sigma^\#\left(\alpha^\#(
\|\theta\|_{\infty,\Omega} +\|\phi\|_{\infty,\Omega} )\|\nabla\theta\|_{p,\Omega}
+\|\phi\|_{\infty,\Omega}  \|\nabla\phi\|_{p,\Omega}\right),
 \end{align*}
 with
 \begin{align*}
  \mathcal{M}_3&= 2 \times  5^{1/p} 
\frac{ 2^{2-3/p}(r_\#)^{2/p-1}(\gamma^\#)^{\ell\,'/2}(b_\#)^{-1/[2(\ell-1)]} }
{ (p-1-\upsilon(p-2))^{1/p} \sqrt{k_\#}}; \\ 
  \mathcal{M}_4&= 2\times
   5^{1/p}  \frac {[1+\upsilon (p-1)]^{1/p}\sqrt{2/k_\#+1} }
{ (p-1-\upsilon(p-2))^{1/p}  \sqrt{k_\#}};  \\ 
\mathcal{M}_5&= 2^{3/2}    5^{1/p} \frac{ (r_\#)^{2/p-1}
|\Omega|^{{1\over 2}-{1\over p}}
+[2^{(n+1)(p-2)/2}+\upsilon (p-1)]^{1/p}\sqrt{1+k_\#}}
{(p-1-\upsilon(p-2))^{1/p}  k_\#}.
  \end{align*}
  
By the other hand,  (\ref{supess}) implies that
\[ 
{\rm  ess } \sup_{\Omega}|\Theta|\leq 1
+
\mathcal{Z}_1(k_\#,b_\#) \sigma^\#C_\infty  \mathcal{M}_6 R^2+
\mathcal{Z}_2(k_\#,b_\#)\gamma^\#
\|\theta_\mathrm{e}\|_{(\ell-1)p,\Gamma}^{\ell-1} ,
\]
with $ \mathcal{M}_6=
 \alpha^\#(1+\mathcal{M}_1+\mathcal{M}_2)
+(\mathcal{M}_1+\mathcal{M}_2)^2.$

 Hence, we conclude
 \begin{align}
  \|\nabla\Theta\|_{p,\Omega}+\|\Theta\|_{\ell,\Gamma}\leq 
 \mathcal{M}_3
\|\theta_\mathrm{e}\|_{\ell,\Gamma}^{\ell/2}
 +\mathcal{M}_4  \gamma^\#
\|\theta_\mathrm{e}\|_{(\ell-1)p,\Gamma}^{\ell-1}
 + \left({ \gamma^\#\over b_\#}\right)^{1/(\ell-1)}\|\theta_\mathrm{e}
\|_{\ell,\Gamma} + \nonumber\\
+2^{\ell-2}b^\#\mathcal{M}_4  |\Gamma|^{1/p}
\left(1+\mathcal{Z}_2(k_\#,b_\#)\gamma^\#
\|\theta_\mathrm{e}\|_{(\ell-1)p,\Gamma}^{\ell-1} 
\right)^{\ell-1}+ \nonumber\\
 +2^{\ell-2}b^\#\mathcal{M}_4  |\Gamma|^{1/p}
\left( 
\mathcal{Z}_1(k_\#,b_\#) \sigma^\#C_\infty  \mathcal{M}_6 \right)^{\ell-1}
R^{2(\ell-1)} +\nonumber \\ 
+\left({\ell\,' |\Omega|^{1-{2\over p}}\over 2b_\# k_\#}\right)^{1/\ell}
(\sigma^\#C_\infty)^{2/\ell}\mathcal{M}_6^{2/\ell}R^{4/\ell}
+C_\infty\sigma^\# \mathcal{M}_5 \mathcal{M}_6 R^2 = \mathcal{Q}(R), 
\qquad \label{defQ}
 \end{align}
taking $\ell\geq 2$ into account.
 Considering that $\mathcal{Q}(0)>0$, the smallness condition $\mathcal{Q}(1)<1$
 assures that the continuous function
$\mathcal{Q}-I$, with $I$ denoting the identity function, has one positive real root
 $R< 1$, by the application of the intermediate theorem, {\em i.e.} $\mathcal{Q}(R)=R$.

Our aim is
 to find a fixed point $\theta=\mathcal{T}(\theta)$, by applying the following
Tychonoff extension to weak topologies of the Schauder fixed point
theorem \cite[pp. 453-456 and 470]{dsch}:
\begin{theorem}\label{fpt} 
Let $K$ be a nonempty   compact convex subset of a
locally convex  space $X$. Let
${\mathcal T}:K\rightarrow K$ be a continuous 
operator. Then $\mathcal T$ has at least one fixed point.
\end{theorem}

Set the reflexive Banach space
 $X=V_{p,\ell}$ $(p>n=2)$ endowed with the weak topology, and $K=\overline B_R(0)$.
 It remains to prove that the operator $\mathcal{T}$
is continuous (for the weak topologies).
Let  $\{\theta_m\}_{m\in\mathbb{N}}$ be a sequence in $ {K}$ such that
  $\theta_m\rightharpoonup\theta$ in $W^{1,p}(\Omega)$,
  and $\Theta_m$ and $\phi_m$ be  the corresponding solutions 
 of (\ref{auxe}) and (\ref{auxt}), respectively.
 Observe that $\theta\in K$ because $K$ is convex and  closed,
hence it is weakly closed in  $W^{1,p}(\Omega)$.
From (\ref{cotapp}) and  $\theta_m,\Theta_m\in K$, there exists $ (\Theta,\phi)
 \in K\times V_{p}$ being
 a weak cluster point of  $\{  (\Theta_m,\phi_m)\}$.  
Let then  $\{  (\Theta_m,\phi_m)\}$ be a
non-relabeled subsequence such that 
  $  (\Theta_m,\phi_m)\rightharpoonup  (\Theta,\phi)$ in $[ W^{1,p}(\Omega)]^{2}$.

By appealing to the  compactness   embedding
  $W^{1,p}(\Omega)\hookrightarrow\hookrightarrow C(\bar\Omega)$ for $p>n=2$,
   we have that 
  $\theta_m\rightarrow\theta$ in $ L^\infty(\Omega)$, and also
pointwisely in $\Omega$.
  Applying the   Krasnoselski Theorem to the Nemytskii operators
$\sigma$ and $\alpha_{\rm s}$, we get
\[
  \sigma(\theta_m)\nabla w\rightarrow\sigma (\theta)\nabla w,\quad
(\alpha_{\rm s}\sigma)(\theta_m)\nabla w\rightarrow
(\alpha_{\rm s} \sigma) (\theta)\nabla w\quad
   \mathrm{in} \quad\mathbf{L} ^{p'}(\Omega),
\] 
making use of the Lebesgue's dominated convergence theorem, 
with (\ref{amm}) and (\ref{smm}).
 
  Then, the variational equality
 (\ref{auxe}) for the solutions  $\phi_m$ 
passes to the limit  as $m$ tends to infinity,
 concluding  (\ref{auxe}).

Similar  strong convergence holds for the leading coefficient $k$.
The  compactness embedding also implies that 
  $\phi_m\rightarrow\phi$ in $ L^\infty(\Omega)$, and  
$\theta_m\rightarrow\theta$ and 
$\Theta_m\rightarrow\Theta$ in $ L^\infty(\partial\Omega)$.
In particular, $\theta_m$
pointwisely converges to $\theta$ on $\Gamma$.
Thus, we can proceed as above, with the aid of (\ref{fmm})-(\ref{gmm}),
by applying the  Lebesgue dominated convergence theorem, obtaining 
 $f_\lambda(\theta_m)|\Theta_m|^{\ell-2}\Theta_mv
 \rightarrow f_\lambda(\theta)|\Theta|^{\ell-2}\Theta v$ and
 $\gamma(\theta_m)v \rightarrow \gamma(\theta) v$ in $L^1(\Gamma)$ and
in $L^{p'}(\Gamma)$, respectively.

Then,  the variational equality
 (\ref{auxt}) for the solutions  $ \Theta_m$
 passes to the limit  as $m$ tends to infinity,
 concluding  (\ref{auxt}).
 
  Therefore, Theorem \ref{fpt} ensures the existence of
at least one fixed point $\theta=\mathcal{T}(\theta)$
 concluding the proof of Theorem \ref{main2}.
 
\section*{Appendix}

Let us prove that $\sigma_s\geq 0$, with $\sigma_s$ representing
the entropy production which verifies
\[ 
\rho\partial_t s=-\nabla\cdot{\bf J}_s+\sigma_s,
\]
where  $\rho$ denotes the density, 
 $s$ denotes the  specific entropy, and ${\bf J}_s$ denotes   the entropy flux.

In the absence of external forces, 
the conservation laws of  energy  and electric charge are, respectively,
\begin{align}\label{dte}
\rho{\partial_t e}=-\nabla\cdot{\bf J};\\
\rho {\partial_t q}=-\nabla\cdot{\bf j}.\label{dtq}
\end{align}
 Here $e$ denotes the specific internal energy,
and  $q$ is the specific electric charge.
 We multiply (\ref{dte}) by $1/\theta$ and (\ref{dtq}) by $\phi/\theta$,
 with 
 $\theta$ denoting the absolute temperature,
and $\phi$ representing the electric potential.
 Gathering the obtained equations with the local form of the Gibbs equation
\cite{liu}:
\[
{\rm d} e=\theta {\rm d}s+\phi{\rm d}q,
\]
we deduce
\[
\left\{\begin{array}{l}
{\bf J}_s=\left({\bf J}-\phi{\bf j}\right)/\theta 
\quad (={\bf q}/\theta)\\
\sigma_s={\bf J}\cdot\nabla({1/ \theta})-
{\bf j}\cdot\nabla (\phi/\theta)
\end{array}\right.
\]
Substituting (\ref{pheno1})-(\ref{pheno2}) into the above expression of
$\sigma_s$ we find
\[
\sigma_s={(\nabla\theta)^\top   k\nabla\theta\over \theta^2}+
{(\alpha_{\rm s}\nabla\theta+\nabla\phi)^\top
\sigma(\alpha_{\rm s}\nabla\theta+\nabla\phi)\over\theta},
\]
if provided by a symmetric matrix $\sigma$, and $k=k_T+\Pi\alpha_{\rm s}\sigma$.
Therefore, we conclude  $\sigma_s\geq 0$, under positive 
semidefinite matrices $k$ and $\sigma$.

\end{document}